\documentclass[12pt,a4paper]{article}%
\usepackage{amsmath}
\usepackage{graphicx}
\usepackage{epsfig}%
\usepackage{amsfonts}%
\usepackage{amssymb}
\usepackage{bm}
\usepackage{color}

\newtheorem{theorem}{Theorem}[section]

\newtheorem{definition}[theorem]{Definition}
\newtheorem{example}[theorem]{Example}
\newtheorem{lemma}[theorem]{Lemma}
\newtheorem{proposition}[theorem]{Proposition}

\newtheorem{conjecture}[theorem]{Conjecture}
\newenvironment{proof}[1][Proof]{\noindent \emph{#1.} }{\hfill \ 
\rule{0.5em}{0.5em}}

\makeatletter\@addtoreset{equation}{section}\makeatother
\makeatletter\@addtoreset{figure}{section}\makeatother
\makeatletter\@addtoreset{table}{section}\makeatother
\def\rank{\mathop{\mathrm{rank}}\nolimits}

\DeclareMathOperator*{\diag}{diag}

\newcommand{\MTran}[1]{{#1}'}

\usepackage{tikz}
\usetikzlibrary{calc}
\usetikzlibrary{shapes}
\usetikzlibrary{fit}

\usepackage{tikz}
\usetikzlibrary{calc}
\usetikzlibrary{shapes}
\usetikzlibrary{fit}

\def\blszh{0.3}
\def\blszv{0.3}
\def\lnlen{0.6}

\def\genblock[#1,#2,#3]{
\coordinate (z) at ($(z)-(#2*\blszh,#3*\blszv)$);
\draw[semithick,fill=white] ($(z)+(\blszh,-\blszv)$) -- ($(z)+(-\blszh,-\blszv)$) -- ($(z)+(-\blszh,\blszv)$);
\draw[semithick,fill=white] ($(z)+(\blszh,-\blszv)$) -- ($(z)+(\blszh,\blszv)$) -- ($(z)+(-\blszh,\blszv)$);
\node at (z) {#1};
}
\def\activeblock[#1,#2,#3]{
\coordinate (z) at ($(z)-(#2*\blszh,#3*\blszv)$);
\draw[very thick,color=red,fill=white] ($(z)+(\blszh,-\blszv)$) -- ($(z)+(-\blszh,-\blszv)$) -- ($(z)+(-\blszh,\blszv)$);
\draw[very thick,color=red,fill=white] ($(z)+(\blszh,-\blszv)$) -- ($(z)+(\blszh,\blszv)$) -- ($(z)+(-\blszh,\blszv)$);
\node at (z) {#1};
}
\def\leftorthblock[#1,#2,#3]{
\coordinate (z) at ($(z)-(#2*\blszh,#3*\blszv)$);
\draw[semithick,fill=blue!35] ($(z)+(\blszh,-\blszv)$) -- ($(z)+(-\blszh,-\blszv)$) -- ($(z)+(-\blszh,\blszv)$);
\draw[semithick,fill=white] ($(z)+(\blszh,-\blszv)$) -- ($(z)+(\blszh,\blszv)$) -- ($(z)+(-\blszh,\blszv)$);
\node at (z) {#1};
}
\def\rightorthblock[#1,#2,#3]{
\coordinate (z) at ($(z)-(#2*\blszh,#3*\blszv)$);
\draw[semithick,fill=white] ($(z)+(\blszh,-\blszv)$) -- ($(z)+(-\blszh,-\blszv)$) -- ($(z)+(-\blszh,\blszv)$);
\draw[semithick,fill=blue!35] ($(z)+(\blszh,-\blszv)$) -- ($(z)+(\blszh,\blszv)$) -- ($(z)+(-\blszh,\blszv)$);
\node at (z) {#1};
}
\def\rdorthblock[#1,#2,#3]{
\coordinate (z) at ($(z)-(#2*\blszh,#3*\blszv)$);
\draw[semithick,fill=white] ($(z)+(-\blszh,-\blszv)$) -- ($(z)+(-\blszh,\blszv)$) -- ($(z)+(\blszh,\blszv)$);
\draw[semithick,fill=blue!35] ($(z)+(-\blszh,-\blszv)$) -- ($(z)+(\blszh,-\blszv)$) -- ($(z)+(\blszh,\blszv)$);
\node at (z) {#1};
}
\def\luorthblock[#1,#2,#3]{
\coordinate (z) at ($(z)-(#2*\blszh,#3*\blszv)$);
\draw[semithick,fill=blue!35] ($(z)+(-\blszh,-\blszv)$) -- ($(z)+(-\blszh,\blszv)$) -- ($(z)+(\blszh,\blszv)$);
\draw[semithick,fill=white] ($(z)+(\blszh,\blszv)$) -- ($(z)+(\blszh,-\blszv)$) -- ($(z)+(-\blszh,-\blszv)$);
\node at (z) {#1};
}

\def\leftline[#1,#2]{
\path[sloped,semithick,-] ($(z)-(\blszh,-#2*\blszv)-(\lnlen,0)$)
edge [out= 0, in= 180] node[above,color=black] {#1}
($(z)-(\blszh,-#2*\blszv)$);
\coordinate (znew) at ($(z)-(\blszh,-#2*\blszv)-(\lnlen,0)$);
}
\def\rightline[#1,#2]{
\path[sloped,semithick,-] ($(z)+(\blszh,#2*\blszv)$)
edge [out= 0, in= 180] node[above,color=black] {#1}
($(z)+(\blszh,#2*\blszv)+(\lnlen,0)$);
\coordinate (znew) at ($(z)+(\blszh,#2*\blszv)+(\lnlen,0)$);
}
\def\bottomline[#1,#2]{
\path[semithick,-] ($(z)-(-#2*\blszh,\blszv)-(0,\lnlen)$)
edge [out= 90, in= -90] node[left,color=black] {#1}
($(z)-(-#2*\blszh,\blszv)$);
\coordinate (znew) at ($(z)-(-#2*\blszh,\blszv)-(0,\lnlen)$);
}
\def\topline[#1,#2]{
\path[semithick,-] ($(z)+(#2*\blszh,\blszv)$)
edge [out= 90, in= -90] node[left,color=black] {#1}
($(z)+(#2*\blszh,\blszv)+(0,\lnlen)$);
\coordinate (znew) at ($(z)+(#2*\blszh,\blszv)+(0,\lnlen)$);
}

\title{Tensor Numerical Methods for High-dimensional PDEs: 
Basic Theory and Initial Applications
}

    \author{
BORIS N. KHOROMSKIJ\thanks{Max-Planck-Institute for
        Mathematics in the Sciences, Inselstr.~22-26, D-04103 Leipzig,
        Germany ({\tt bokh@mis.mpg.de, http://personal-homepages.mis.mpg.de/bokh/}).}
}

\begin{document}
\date{}

\maketitle

\begin{abstract} 
We present a brief survey on the  modern tensor numerical methods 
for multidimensional stationary and time-dependent partial differential equations (PDEs).
The guiding principle of the tensor approach
is the rank-structured separable approximation of multivariate functions and 
operators represented on a grid. 
Recently, the traditional Tucker, canonical, and matrix product states (tensor train) 
tensor models have been applied to the grid-based electronic structure calculations, 
to parametric PDEs, and to dynamical equations arising in scientific computing.  
The essential progress is based on the quantics tensor approximation  
method  proved to be capable to represent (approximate) function related $d$-dimensional 
data arrays  of size $N^d$ with log-volume complexity, $O(d \log N)$.
Combined with the traditional numerical schemes, these novel tools establish a 
new promising approach 
for solving multidimensional integral and differential equations using low-parametric
rank-structured tensor formats.
As the main example, we describe the  grid-based tensor numerical approach
for solving the 3D nonlinear Hartree-Fock eigenvalue problem, that was the starting point for 
the developments of tensor-structured numerical methods
for large-scale computations in solving real-life multidimensional problems.
We also discuss a new method for the fast 3D lattice summation of 
electrostatic  potentials by assembled low-rank tensor approximation capable to 
treat the potential sum over millions of atoms in few seconds.
We address new results on tensor approximation of the 
dynamical Fokker-Planck and master equations in many dimensions up to $d=20$.
Numerical tests demonstrate the benefits of the rank-structured tensor approximation 
on the aforementioned examples of multidimensional PDEs. In particular, 
the use of grid-based tensor representations in the reduced basis of atomics orbitals 
yields an accurate solution of the Hartree-Fock equation on large $N\times N \times N$ grids
with a grid size of up to $N= 10^{5}$.
\end{abstract}

\section{Introduction}\label{sec:Introd}

In the recent years, the tensor numerical methods  were
recognized  as the basic tool to render  numerical simulations
in higher dimensions tractable. The guiding principle of the tensor numerical methods
is the reduction of the computational process onto a low parametric rank-structured manifold 
by using an approximation of multivariate functions and operators, that relies
on a certain separation of variables. Possible applications
of tensor numerical methods include high-dimensional problems arising in material- and 
bio-sciences, computational quantum chemistry, stochastic modeling and 
uncertainty quantification, dynamical systems, machine learning, financial mathematics, etc. 

In the following discussion a tensor of order $d$ and mode size $N$, or briefly
$N$-$d$ tensor, is considered as a function
on a $d$-fold product index set, $\textsf{\textbf{A}}:I^{\otimes d}\to \mathbb{R}$ with
$I^{\otimes d}=I\times \cdot \cdot\cdot \times I$, and $I=\{1,...,N\}$. 
In the traditional grid-based numerical techniques for $d$-dimensional PDEs  
the parameter $N$ can be associated with the univariate grid size.
The representation of arising $N$-$d$ tensors (coefficient vectors)   
requires a storage size that is exponential in $d$, $N^{d}$,
which causes severe computational difficulties, often called  
the ``curse of dimensionality`` \cite{bellman-dyn-program-1957}.



A class of methods which lead to linear scaling in the dimension are
distinctly linked with the principle of separation of variables.
The multi-linear tensor decompositions based on the Tucker and canonical models 
have long since been used  as a tool
for data processing in the computer science community (e. g. PCA type methods), 
and applied to multidimensional experimental data in chemometrics,
psychometrics and in  signal processing, see a comprehensive bibliography in \cite{KoldaB:07}.
The remarkable approximating properties of the Tucker and canonical decomposition
for wide classes of function related tensors  were revealed in \cite{Khor:06,KhBoVe1:06},
promoting its usage as a tool for the numerical treatment of multidimensional problems in
numerical analysis. 
An introductory description of traditional tensor formats 
with the focus on tensor-structured numerical methods for the calculation of multidimensional
functions and operators is presented in \cite{VeKh_Diss:10,KhorSurv:10,KhorLecZuri2010:10}.
Moreover, the canonical tensor representations obtained by $\mbox{sinc}$ approximations
on a class of analytic multivariate functions
have been proven to provide fast exponential convergence in the separation rank 
\cite{GaHaKh4:02,HaKhtens:04I,Khor:06,Yserent:00}. 


First successes in the rank-structured tensor calculations  of
multivariate functions and operators in the Hartree-Fock equation
originated the grid-based tensor numerical methods in scientific computing
\cite{KhKh3:08,KhKhFl_Hart:09,VeKh_Diss:10,KhVBSchn_TEI:12,VeKhorMP2:13,VeBoKh:Ewald:14,VeKhorCorePeriod:14,ManbyCh:11}.
Combined with the matrix product states (MPS) techniques developed in the physics community,
\cite{White:93,Cirac_TC:04,WangTho:03,Vidal2003,PerVerWoCirac:07}, 
including its particular form, the tensor train (TT) format \cite{OsTy_TT:09,Osel-TT:09},
and with the newly developed quantized  tensor approximation of discretized 
functions \cite{KhQuant:09} and operators \cite{Osel-TT-LOG:09},
these methods boiled up to a powerful tool for the numerical analysis 
in higher dimensions. Concerning computational quantum chemistry, 
the real space numerical methods combined with FEM or plane waves approximations 
have become attractive in (post) Hartree-Fock and DFT calculations as the possible alternative
to traditional approaches 
\cite{LaPySu:1986,HaFaYaBeyl:04,BiVale:11,CanEhrMad:2012,Ortn:ArX,BGKh:12,Frediani:13,RahOsel:13}. 

Literature surveys on the most frequently used tensor formats can be found in 
\cite{KhorSurv:10,Scholl:11,Hack_Book12,GraKresTo:13}. 
In addition, methods of multilinear algebra and nonlinear tensor approximation 
have been discussed, see
\cite{AbMaSe:Book08,KhKh3:08,OsTy_TT:09,KhorLecZuri2010:10,KhQuant:09,GrasHTUCK:09}
and references therein.
The numerical cost of basic multilinear algebra operations on formatted
$N$-$d$ tensors usually
scales linearly in $d$, but could be polynomial in the mode size $N$. 
This leads to essential limitations, since the high precision numerical
simulations might require $N^{\otimes d}$-grids with large mode size $N\sim 10^4,...,10^6$.

The new paradigm of the quantics-TT (QTT) approximation for a class of discretized functions, 
as introduced and rigorously justified in  \cite{KhQuant_P:09,KhQuant:09}, leads to a data compression 
with  $O(\log N)$ complexity scaling.
In this way the QTT approximation applies to the quantized image of the target discrete function,
obtained by its isometric folding transformation to
the higher dimensional {\it quantized tensor space}. 
For example, a vector of size $N=q^L$ ($q=2,3,...$) can be successively 
reshaped by a $q$-adic folding to an $L$-fold tensor in $\bigotimes_{j=1}^L \mathbb{R}^{q}$
of the irreducible mode size $n=q$ (quantum of information), then
the  low-rank approximation in the canonical or TT  
formats can be applied consequently. 

The principal question arises: how can the folding of a vector to a higher dimensional 
tensor in $\bigotimes_{j=1}^L \mathbb{R}^{q}$
lead  to the essential data compression using representations in low-rank tensor formats
and, hence, to the efficient $O(\log N)$-approximation method? 
The constructive answer is formulated by the QTT approximation theory  
proven for the basic classes of function related tensors: 
The TT-rank of quantized exponential, trigonometric, and polynomial $N$-vectors 
remains constant, that is independent of $N$  \cite{KhQuant_P:09,KhQuant:09}.
As a simple corollary, it was shown that the QTT approximation provides exponential
convergence in the QTT-rank for a class of  
analytic function related $N$-vectors and  $N$-$d$ tensors,
which are well representable in terms of the aforementioned functional classes.  
These beneficial approximation features allow to understand why the computation in 
quantized tensor spaces may lead to a log-volume complexity, $O(d \log N)$.
Further QTT approximation results for functional  vectors can be found in 
\cite{KhOsel_3_DMRG:10,Gras_Quant:10,OselQTTfunc:10,DoKhOsel:11,OselTyr_WT:11,DoKh_QTTT:12,VeKhSchn_QTTHF:11,KazReichSchw:13,VeBoKh:Ewald:14}.

The important point for solving PDEs is that typical integral and elliptic differential operators
also admit the low QTT-rank representation.
It was found in \cite{Osel-TT-LOG:09} by numerical tests 
that in some cases the dyadic reshaping of an $2^L \times 2^L$ matrix  leads to a small
TT-rank of the reshaped operator.
An explicit low-rank QTT representation of the matrix exponential as well as  
the discrete quantized Laplacian and its 
inverse were derived in \cite{KhOsel_1_P:09,KhOsel_1:09,KazKhor_1:10}. 


We summarize that the idea of quantized tensor approximation in $d\log N$-complexity
allows a fast functional and operator calculus on very large discretization grids  
(practically unlimited grid size). This opens a way to the efficient computation  
of $d$-dimensional integrals with controllable accuracy, the representation of 
discrete elliptic operators and their inverse, as well as various linear and bilinear mappings
on tensors of order $d$, thus making the decisive step toward tractable numerical methods
for multi-dimensional PDEs.


As a matter of fact the concept of low-rank separable representation of multidimensional functions 
and operators in combination with properly modified traditional numerical schemes,
and with well developed algebraic tools for formatted tensors
has penetrated into the new branch of numerical analysis 
in the form of tensor-structured numerical methods (TNM) for solving multi-dimensional PDEs.
The main ingredients of TNMs for PDEs include:
\begin{itemize}
\item FEM or spectral discretization of $d$-dimensional PDEs in tensor product Hilbert spaces
\item Numerical multilinear algebra of rank-structured tensors
\item Approximate grid-based tensor calculus of multivariate functions and operators in low-parametric 
tensor formats 
\item Quantized tensor approximation
\item Tensor truncated iterative methods for solving discrete
systems of stationary and  time dependent $d$-dimensional PDE. 
\end{itemize}

This survey is merely an attempt 
to outline the main results on TNMs for multidimensional PDEs 
obtained  in the recent years by the author and in collaborations,
which have been presented at the Summer School CEMRACS-2013, 
``Modeling and simulation of complex systems: Stochastic and deterministic approaches'',  
CIRM, Marseille-Luminy, France, 22-26.07.2013.
We mainly focus on the recently developed grid-based tensor numerical methods
for solving the 3D nonlinear Hartree-Fock eigenvalue problem,
including the new method of fast lattice summation of electrostatic  potentials
by the assembled low-rank tensor approximation, and also discuss 
the tensor method of simultaneous $(x,t)$-approximation in TT/QTT formats for the 
dynamical Fokker-Planck and master equations in many dimensions up to $d=20$.

Several important applications of tensor numerical methods
remain beyond the scope of this review. They include  
multidimensional preconditioning \cite{KhorCA:09,AndrTobl:12} and solution of linear systems 
\cite{tobkress-krylov-2010,DoOs-dmrg-solve-2011,USchm_ALS:12,DoSav_AMEn:13,ds-amr1-2013,benner-lowrank-lyapunov-2013}, 
 parametric/stochastic PDEs 
\cite{MattKee:05,KhSchw1:09,KhOsel_2_SPDE:10,DolKhLitMat:14,benner-spde-2013}, 
greedy algorithms 
\cite{BrisMaday-greedy_pdes-2009,Binev-conv_rate_greedy-2011,Cances-greedy-2011,Sueli-greedy_FP-2011},
 PGD model reduction methods \cite{Ammar-cme-2011,FalNouy:12,TamMaitNouy:12}, 
super-fast FFT, convolution and wavelet transforms \cite{DoKhSav:11,KazKhTyrt_Toepl:11,KhMiao:13}, 
tensor-product interpolation and cross approximation 
\cite{OsTy_TT2:09,bebe-aca-2011,so-dmrgi-2011proc,lars-htcross-2013}, 
integration of singular or highly oscillating functions \cite{KhSautVeit:11}, 
control problems \cite{stoll-lowrank-time-2013}, 
complexity theory \cite{BachDahm:13,DahmDevGrSul:14}, etc. 

The rest of the manuscript is structured as follows. Section \ref{sec:TS-formats} introduces the 
basic rank-structured tensor formats
with a focus on the quantized tensor representation of functions and operators.
Section \ref{sec:QTT-solvers} describes the main building blocks in the construction 
and justification of TNMs for the solution of the Hartree-Fock equation  
in {\it ab initio} electronic structure calculations, and 
for the tensor approximation of the real-time high-dimensional stochastic multiparticle 
dynamics governed by master equations. Section \ref{sec:Conclusion} concludes the paper
and outlines challenging problems for future research.

\section{Rank-structured tensor approximation}\label{sec:TS-formats}

In this section we present a short
description of the basic additive and multiplicative tensor formats. 
These formats allow low-parametric function and operator representations by a nonlinear 
mapping onto the rank-structured tensor manifolds.

\subsection{Basic tensor formats for representation of functions and operators}
\label{ssec:Tformat}

A tensor of order $d$, further called $N$-$d$ tensor, is defined as an element 
of finite dimensional {\it tensor-product Hilbert space} 
of the $d$-fold, $N_1\times ... \times N_d$ real/complex-valued arrays,
$\mathbb{W}_{\bf n}\equiv \mathbb{W}_{{\bf n},d}= \bigotimes_{\ell=1}^d X_\ell$,
where $X_\ell=\mathbb{R}^{N_\ell}$ or $X_\ell=\mathbb{C}^{N_\ell}$, and ${\bf n}=(N_1,...,N_d)$. 
An element of $\mathbb{W}_{\bf n}$  can be represented entrywise by
$$
\textsf{\textbf{A}}=[\textsf{A}(i_1,...,i_d)]\equiv [\textsf{A}_{i_1,...,i_d}]
\quad \mbox{with}\quad i_\ell\in I_\ell:=\{1,...,N_\ell\}.
$$
We confine ourselves to the case of real-valued tensors in 
$\mathbb{W}_{\bf n}=\mathbb{R}^{\cal I}$, ${\cal I}=I_1\times ... \times I_d$, 
though all the constructions can be extended to the case of complex-valued 
tensor-product Hilbert space, $\mathbb{W}_{{\bf n}}=\mathbb{C}^{\cal I}$.
The Euclidean scalar product, $\left\langle \cdot ,\cdot \right\rangle :
\mathbb{W}_{\bf n}\times \mathbb{W}_{\bf n}\to \mathbb{R}$,
is defined by
$$
\left\langle \textsf{\textbf{A}} , \textsf{\textbf{B}} \right\rangle:=
\sum_{{\bf i}\in {\cal I}} \textsf{A}({\bf i}) \textsf{B}({\bf i}), \quad
\textsf{\textbf{A}},\textsf{\textbf{B}}\in \mathbb{W}_{\bf n},
$$ 
that merely identifies tensors with multi-indexed Euclidean vectors.
Any particular tensor can be associated with a function of a discrete variable, 
$\textsf{\textbf{A}}: I_1\times ... \times I_d \to \mathbb{R}$.
For ease of presentation, we often consider the equal-size tensors,
i.e. $I=I_\ell=\{1,...,N\}$ ($\ell=1,...,d$) with the short notation ${\cal I}=I^{\otimes d}$.

The storage size for $N$-$d$ tensor scales exponentially in $d$, 
$\operatorname{dim}(\mathbb{W}_{{\bf n},d})=N^d$ (the "curse of dimensionality")
that makes the traditional numerical methods, 
characterized by linear complexity scaling in the discrete problem size, non-tractable
already for moderate $d$.
The tensor numerical methods are based on the idea of a low-rank separable tensor 
decomposition (approximation) applied 
to all discretized functions and operators describing the physical model governed, 
say, by partial differential equation (PDE).
In this way, the simplest separable elements are given by rank-$1$ tensors,
\[
 \textsf{\textbf{A}}= \bigotimes_{\ell=1}^d \textsf{A}^{(\ell)},
\quad \textsf{A}^{(\ell)}\in \mathbb{R}^{N_\ell},
\]
which can be stored with $d N$ numbers (parameters). 
 in the recent decades different classes of rank-structured tensor representations 
 (formats) have been introduced in the literature. 
 The important feature of such formats
is that the respective ``formatted`` elements are represented with a small number of
parameters that scales linearly in the dimension $d$. Each of these low-parametric 
formats include rank-$1$ tensors as the ''simplest`` elements. All these formats
can be viewed as multidimensional generalizations of the notion of a rank-$R$ matrix.

The basic commonly used separable representations of tensors are
described by the canonical and Tucker formats.
We say that an element $\textsf{\textbf{A}}\in \mathbb{W}_{\bf n}$ belongs to
the class of $R$-term canonical tensors if it has the representation
\begin{equation}\label{eqn:canTens}
\textsf{\textbf{A}}= \sum_{\alpha=1}^R \bigotimes_{\ell=1}^d \textsf{A}^{(\ell)}_\alpha,
\quad \textsf{A}^{(\ell)}_\alpha \in \mathbb{R}^{N_\ell},
\end{equation}
or in index notation
\begin{equation*}\label{par:can}
\textsf{A}(i_1,\ldots,i_d) = \sum_{\alpha=1}^R \textsf{A}^{(1)}_\alpha(i_1) \cdots 
\textsf{A}^{(d)}_\alpha(i_d),
\quad [\textsf{A}^{(\ell)}_\alpha(\cdot)] \in \mathbb{R}^{N_\ell}.
\end{equation*}
Now the storage cost is bounded by  $d R N$. For $d\geq 3$
computation of the canonical rank of the tensor $\textsf{\textbf{A}}$, i.e.
the minimal number $R$ in representation (\ref{par:can}) and the respective
decomposition, is an $N$-$P$ hard problem.
In the case $d=2$ the representation (\ref{eqn:canTens}) is merely a rank-$R$ matrix.

We say that
$\textsf{\textbf{A}}\in \mathbb{W}_{\bf n}$ belongs to the rank
$\mathbf{r} = [r_1,\ldots,r_d]$ Tucker format \cite{Tuck:66,DMV-SIAM:00} if there 
exists a representation
\begin{equation*}\label{par:tucker}
\textsf{A}(i_1,\ldots,i_d) = 
\sum\limits_{\alpha_1=1}^{r_1} \cdots \sum\limits_{\alpha_d =1}^{r_d}
\textsf{B}_{\alpha_1,\ldots,\alpha_d} \textsf{A}^{(1)}_{\alpha_1}(i_1) \cdots \textsf{A}^{(d)}_{\alpha_d}(i_d), 
\quad \textsf{A}^{(\ell)}_{\alpha_\ell}(\cdot)\in \mathbb{R}^{N_\ell}.
\end{equation*}
In this case the storage cost is bounded by $drN +r^d$, $r=\max r_\ell$, where the second term 
estimates the size of the Tucker core tensor 
$\textsf{\textbf{B}}=[\textsf{B}_{\alpha_1,\ldots,\alpha_d}]\in \mathbb{R}^{r_1\times ... \times r_d}$.
Without loss of generality, the set of Tucker vectors 
$\{\textsf{A}^{(\ell)}_{\alpha_\ell}\}$ with
$\textsf{A}^{(\ell)}_{\alpha_\ell} \in \mathbb{R}^{N_\ell}$
($\ell=1,...,r_\ell$) can be orthogonalized. 
In the case $d=2$ the orthogonal Tucker decomposition is equivalent to the singular 
value decomposition (SVD) of a rectangular matrix.

The important generalization to the Tucker 
representation is the two-level Tucker-canonical format \cite{Khor1:08,KhKh3:08} 
that consists of Tucker tensors with the core array $\textsf{\textbf{B}}$ represented in the 
low-rank canonical form.

The product-type representation of $d$th order tensors, which is
called   the matrix product states (MPS) decomposition in the physical literature,
was introduced and successfully applied in DMRG quantum computations
\cite{White:93,Vidal2003,Cirac_TC:04},
and, independently, in quantum molecular dynamics 
as the multilayer (ML) MCTDH methods \cite{WangTho:03,MeyGaWo_book:09,LubichBook:08}.
Representations by MPS type formats reduce the complexity of storage to $O(d r^2 N)$,
where $r$ is the maximal rank parameter. 

In the recent years the various versions of the MPS-type
tensor format were discussed and further investigated in mathematical literature 
including the hierarchical dimension splitting \cite{Khor:06}, 
the tensor train (TT) \cite{OsTy_TT:09,Osel-TT:09}, the tensor chain (TC) and
combined Tucker-TT \cite{KhQuant:09}, the QTT-Tucker \cite{DoKh_QTTT:12} formats,
as well as the  hierarchical Tucker  representation \cite{Hack_Book12} that
belongs to the class of ML-MCTDH methods \cite{WangTho:03}, or more generally tensor network 
states models.
The MPS-type tensor approximation was
proved by extensive numerics to be efficient in high-dimensional electronic/molecular structure
calculations, in molecular dynamics and in quantum information theory 
(see survey papers \cite{Cirac_TC:04,HuklWaSch:10,KhorSurv:10,Scholl:11}).

The TT format that is the particular case of MPS type factorization in the case of 
open boundary conditions, can be defined as follows.
For a given rank parameter ${\bf r}=(r_0,...,r_d)$, and the respective
index sets $J_\ell=\{1,...,r_\ell\}$ ($\ell=0,1,...,d$),
with the constraint $ J_0= J_d=\{1\}$ (i.e., $r_0=r_d$),
the rank-${\bf r}$ TT format contains all elements
$\textsf{\textbf{A}}=[\textsf{A}(i_1,...,i_d)]\in\mathbb{W}_{\bf n} $
which
can be represented as the contracted products of $3$-tensors 
over the $d$-fold product index set ${\cal J}:=\times_{\ell=1}^d J_\ell$, such that
\[
\textsf{\textbf{A}}=\sum\limits_{{\bf \alpha}\in {\cal J}} \textsf{A}^{(1)}_{\alpha_1}
\otimes \textsf{A}^{(2)}_{\alpha_1,\alpha_2} \otimes 
\cdots \otimes \textsf{A}^{(d)}_{\alpha_{d-1}}
\equiv \textsf{A}^{(1)}\Join \textsf{A}^{(2)}\Join \cdots \Join \textsf{A}^{(d)}, 
\]
where
$\textsf{A}^{(\ell)}_{\alpha_\ell,\alpha_{\ell+1}}   \in \mathbb{R}^{N_\ell}$, 
($\ell=1,...,d$), and 
$\textsf{A}^{(\ell)}=[\textsf{A}^{(\ell)}_{\alpha_\ell,\alpha_{\ell+1}}]$
is the vector-valued $r_\ell\times r_{\ell+1}$ matrix ($3$-tensor).
Here and in the following (e.g. in Definition \ref{def:OTT_Matformat}) 
the rank product operation ``$\Join$'' is defined as a regular 
matrix product of the two core matrices, their fibers (blocks) being multiplied 
by means of tensor product \cite{KazKhor_1:10}. In the index notation we have
$$
\textsf{A}(i_1,...,i_d)=\sum\limits_{\alpha_1 =1}^{r_1} \cdots  \sum\limits_{\alpha_d =1}^{r_d}
\textsf{A}^{(1)}_{\alpha_1}(i_1)  \textsf{A}^{(2)}_{\alpha_1,\alpha_2}(i_2)
\cdot\cdot\cdot \textsf{A}^{(d)}_{\alpha_{d-1}}(i_d)
\equiv \textsf{A}^{(1)}({i_1}) \textsf{A}^{(2)}({i_2}) \ldots \textsf{A}^{(d)}({i_d}),
$$
such that the latter is written in the {\it matrix product} form
(explaining the notion MPS) 
where $\textsf{A}^{(\ell)}(i_\ell) $ is $r_{\ell-1} \times r_\ell$ matrix.

\begin{example}\label{exm:figMPS}
 Figure \ref{fig:MPS_format} illustrates the TT representation of a $5$th-order tensor: each
 particular entry indexed by $(i_1,i_2,...,i_5)$ is factorized as a product of five matrices,
 thus explaining the name MPS.
\end{example}

\begin{figure}[tbh]\label{fig:MPS_format}
  \begin{center}
\includegraphics[width=11.0cm]{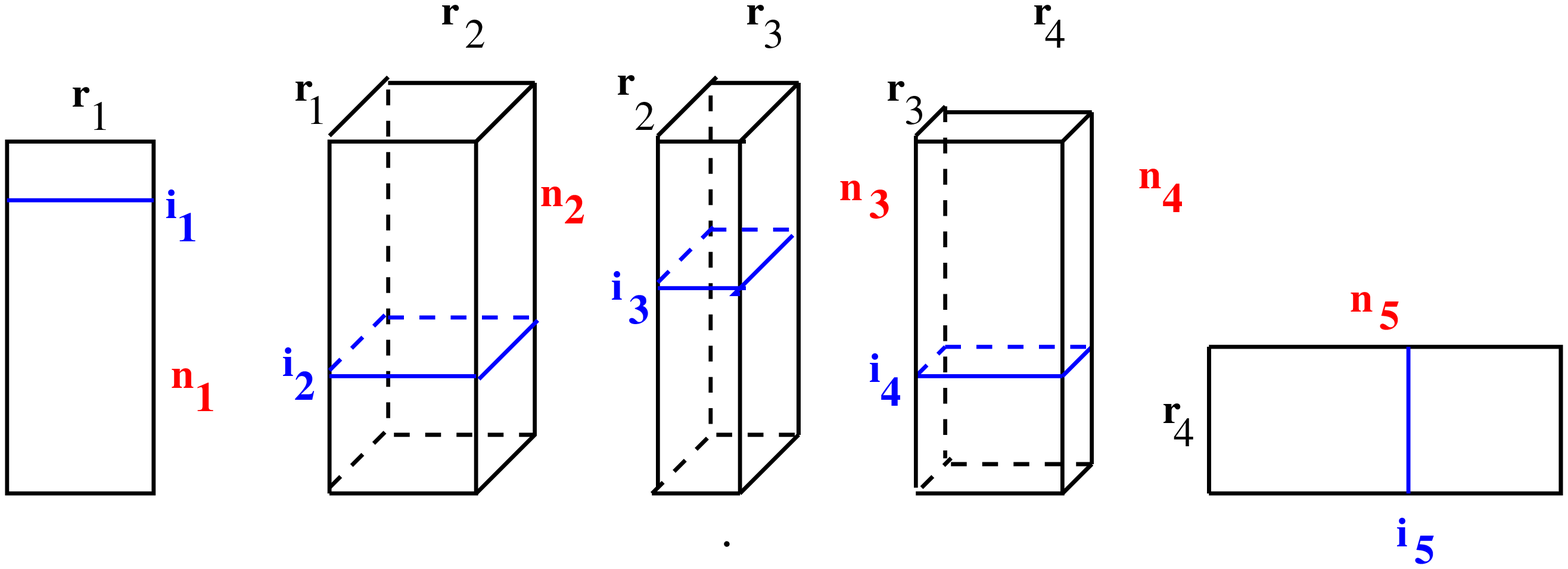}
\end{center}
\caption{Visualizing $5$th-order TT tensor.}
\end{figure}

In case $J_0 = J_d \neq \{1\}$, we arrive at the more general form of MPS,
the so-called tensor chain (TC) format. 
In some cases TC tensor can be represented as a sum of not more than $r_\ast$ TT-tensors
($r_\ast=\max {r_\ell}$), which can be converted to the TT tensor based on multilinear algebra
operations like sum-and-compress.
The storage cost for both TC and TT formats is bounded by $O(d r^2 N)$, $r=\max {r_\ell}$.

Clearly, one and the same tensor might have different ranks in different formats
(and, hence, different number of representation parameters).
Next example considers the Tucker and TT representations of a {\it function related} 
canonical tensor $\textsf{\textbf{F}}:=T(f)$ obtained by sampling of
the function $f(x)=x_1+...+x_d$, $x\in [0,1]^d$, on the Cartesian grid
of size $N^{\otimes d}$ and specified by $N$-vectors $X_\ell=\{ih\}_{i=1}^N$, ($h=1/N$, $\ell=1,...,d$)
and all-ones vector $1\in \mathbb{R}^N$.
The canonical rank of this tensor can be proven to be exactly $d$, \cite{Landsberg_book:2012}.
\begin{example} We have
 $rank_{Tuck}(\textsf{\textbf{F}})={2}$, with the explicit representation
$$
\textsf{\textbf{F}}= \sum\limits_{{\bf k}={\bf 1}}^{\bf 2} b_{\bf k} V^{(1)}_{k_1} \otimes\ldots\otimes 
V^{(d)}_{k_d}, \quad V^{(\ell)}_{1} =1, V^{(\ell)}_{2} = X_\ell, 
\quad [b_{\bf k}] \in \bigotimes_{\ell=1}^d \mathbb{R}^2.
$$
as well as  $rank_{TT}(\textsf{\textbf{F}})={2}$, due to exact decomposition,
$$
\textsf{\textbf{F}} =\begin{bmatrix}
 X_1 & 1       \\
\end{bmatrix} \Join
\left[\begin{array}{c c}   1 & 0 \\  
X_2 & 1 \end{array}\right]\Join ... \Join
\left[\begin{array}{c c}   1 & 0 \\  
X_{d-1} & 1 \end{array}\right] \Join 
\begin{bmatrix}
1  \\
 X_d  \\
\end{bmatrix}. 
$$
\end{example}

The rank-structured tensor formats like canonical, Tucker and MPS/TT-type decompositions 
induce the important concept
of {\it canonical, Tucker or matrix product operators} (CO/TO/MPO) acting between two 
tensor-product Hilbert spaces, each of dimension $d$,
$$
{\bm{\mathcal{A}}}: \mathbb{X}=\bigotimes_{\ell=1}^d X^{(\ell)} \to 
\mathbb{Y}=\bigotimes_{\ell=1}^d Y^{(\ell)}.
$$
For example, the $R$-term canonical operator (matrix) takes a form
\begin{equation*}\label{eqn:canMatr}
{\bm{\mathcal{A}}}= \sum_{\alpha=1}^R \bigotimes_{\ell=1}^d {\cal A}^{(\ell)}_\alpha,
\quad {\cal A}^{(\ell)}:X^{(\ell)} \to Y^{(\ell)}.
\end{equation*}
The action ${\bm{\mathcal{A}}} \textsf{\textbf{X}} $ on rank-$R_X$ canonical tensor 
$\textsf{\textbf{X}} \in \mathbb{X}$
is defined as $R R_X$-term canonical sum in $\mathbb{Y}$,
\[
 {\bm{\mathcal{A}}} \textsf{\textbf{X}}= 
\sum_{\alpha=1}^R \sum_{\beta=1}^{R_X} 
\bigotimes_{\ell=1}^d {\cal A}^{(\ell)}_\alpha \textsf{X}^{(\ell)}_\beta \in \mathbb{Y}.
\]
The rank-${\bf r}$ Tucker matrix can be defined by the similar way.

In the case of rank-${\bf r}$ TT format the respective matrices are defined as follows.
\begin{definition}\label{def:OTT_Matformat} 
The rank-$\bf r$ TT-operator (TTO/MPO) decomposition symbolized by a set of factorized 
operators ${\bm{\mathcal{A}}}$, is defined by
\[
{\bm{\mathcal{A}}}=\sum\limits_{{\bf \alpha}\in {\cal J}} {\cal A}^{(1)}_{\alpha_1} \otimes
{\cal A}^{(2)}_{\alpha_1 \alpha_2} \otimes 
\cdots \otimes {\cal A}^{(d)}_{\alpha_{d-1}} 
\equiv {\cal A}^{(1)}\Join {\cal A}^{(2)}\Join \cdots \Join {\cal A}^{(d)},
\]
where ${\cal A}^{(\ell)}=[A^{(\ell)}_{\alpha_\ell \alpha_{\ell+1}}]$
denotes the operator valued $r_\ell\times r_{\ell+1}$ matrix, and where
${A}^{(\ell)}_{\alpha_\ell \alpha_{\ell+1}}: X^{(\ell)} \to Y^{(\ell)}$, 
($\ell=1,...,d$), or in the index notation 
\begin{eqnarray}\label{Def:TTM}
{\bm{\mathcal{A}}}({i_{1}, j_{1}, \ldots ,i_{d},j_{d}})
& = &
\sum_{\alpha_{1}=1}^{r_{1}}\ldots\sum_{\alpha_{d-1}=1}^{r_{d-1}}
{\cal A}^{(1)}_{\alpha_{1}}(i_{1},j_{1})
{\cal A}^{(2)}_{\alpha_{1}\alpha_{2}}(i_{2},j_{2})\cdot\ldots\cdot
\nonumber \\
& \cdot &
{\cal A}^{(d-1)}_{\alpha_{d-2}\alpha_{d-1}}(i_{d-1},j_{d-1})
{\cal A}^{(d)}_{\alpha_{d-1}}(i_{d},j_{d}).
\end{eqnarray}
\end{definition}

Given rank-${\bf r}_X$ TT-tensor 
$\textsf{\textbf{X}}=\textsf{X}^{(1)}\Join \textsf{X}^{(2)}\Join \cdots 
\Join \textsf{X}^{(d)}\in \mathbb{X}$,
the action ${\bm{\mathcal{A}}} \textsf{\textbf{X}}= \textsf{\textbf{Y}}$ 
is defined as the TT element $\textsf{\textbf{Y}}=\textsf{Y}^{(1)}\Join \textsf{Y}^{(2)}\Join \cdots 
\Join \textsf{Y}^{(d)} \in \mathbb{Y}$,
$$
{\bm{\mathcal{A}}} \textsf{\textbf{X}} = \textsf{Y}^{(1)}\Join \textsf{Y}^{(2)}\Join \cdots 
\Join \textsf{Y}^{(d)}  \in \mathbb{Y},\quad \mbox{with} \quad 
\textsf{Y}^{(\ell)}=
[{\cal A}^{(\ell)}_{\alpha_1 \alpha_2}  \textsf{X}^{(\ell)}_{\beta_1 \beta_2}]_{\alpha_1 \beta_1,\alpha_2 \beta_2},
$$ 
where in the brackets we use the standard matrix-vector multiplication. 
The TT-rank of $\textsf{\textbf{Y}}$ is bounded by ${\bf r}_Y \leq {\bf r}\odot{\bf r}_X $, 
where $\odot$ means the standard Hadamard (entry-wise) product of two vectors.

To describe the index-free operator representation of the TT matrix-vector product,
we introduce the tensor operation denoted by $ \Join^\ast$ that 
can be viewed as dual to $\Join$: it is defined as the tensor (Kronecker) product of the two corresponding core 
matrices, their blocks being multiplied by means of a regular matrix product operation.
Now with the substitution $\textsf{Y}^{(\ell)}={\cal A}^{(\ell)}\Join^\ast \textsf{X}^{(\ell)}$ 
the matrix-vector product in TT format takes the operator form,
\[
 {\bm{\mathcal{A}}} \textsf{\textbf{X}} = ({\cal A}^{(1)}\Join^\ast \textsf{X}^{(1)} ) \Join
\cdots \Join ({\cal A}^{(d)}\Join^\ast \textsf{X}^{(d)} ).
\]

As an example, we consider the finite difference
negative $d$-Laplacian over uniform tensor grid, which is known to have
the Kronecker rank-$d$ representation, 
\begin{equation}\label{dLAPL}
{\Delta}_d = A \otimes I_N\otimes ... \otimes I_N + I_N \otimes A\otimes I_N
...\otimes I_N + ... + I_N \otimes I_N ... \otimes A 
\in \mathbb{R}^{{N}^{\otimes d}\times {N}^{\otimes d}},
\end{equation}
with $A= {\Delta}_1=\operatorname{tridiag}\{-1,2,-1\}\in \mathbb{R}^{N\times N}$, 
and $I_N$ being the $N\times N$ identity matrix.

For the canonical rank we have 
$rank_{Can}({\Delta}_d) = d$, while
the TT-rank of ${\Delta}_d$ is equal to $2$ for any dimension due to the 
explicit representation \cite{KazKhor_1:10},
$$
{\Delta}_d =\begin{bmatrix}
 \Delta_1 & I_N       \\
\end{bmatrix}
\Join
\left[\begin{array}{c c}   I_N & 0 \\  
\Delta_1 & I_N \end{array}\right]^{\otimes\left(d-2\right)}
\Join
\begin{bmatrix}
I_N  \\
\Delta_1  \\
\end{bmatrix}, 
$$
where the rank product operation ``$\Join$'' in the matrix case is defined as above \cite{KazKhor_1:10}.
The similar statement is true concerning the Tucker rank, $rank_{Tuck}({\Delta}_d)=2$.


\subsection{Quantics tensor approximation of function related vectors}\label{ssec:QQTform}


In the case of large mode size, the asymptotic storage cost for a class of function related
 $N$-$d$ tensors can be reduced to $O(d \log N)$ by using quantics-TT (QTT) 
tensor approximation method \cite{KhQuant_P:09,KhQuant:09}.
The QTT-type approximation of an $N$-vector with $N=q^L$, $L\in \mathbb{N}$, is defined as the 
tensor decomposition (approximation) in the canonical, TT or more general formats applied to a tensor 
obtained by the $q$-adic folding (reshaping) of the target vector to an $L$-dimensional 
$q\times ... \times q$ data array (tensor) that is thought as an element of the $L$-dimensional 
quantized tensor space.

                      
In particular, in the vector case, i.e. for $d=1$, a vector 
$
X=[X(i)]_{i\in I}\in \mathbb{W}_{{N},1},
$
is reshaped to its quantics image in 
$
\mathbb{Q}_{q,L}= \bigotimes_{j=1}^{L}\mathbb{K}^{q},
 \; \mathbb{K}\in \{\mathbb{R},\mathbb{C}\},
$ 
by $q$-adic folding, 
\[
\mathcal{F}_{q,L}: X \to \textsf{\textbf{Y}}
=[Y({\bf j})]\in \mathbb{Q}_{q,L}, \quad {\bf j}=\{j_{1},...,j_{L}\}, 
\quad \mbox{with} \quad j_{\nu}\in \{1,2,...,q\}, \,\nu=1,...,L, 
\]
where for fixed $i$, we have $Y({\bf j}):= X(i)$, and $j_\nu=j_\nu(i)$ is defined via $q$-coding,
$
j_\nu - 1= C_{-1+\nu}, 
$
such that the coefficients $C_{-1+\nu} $ are found from the
$q$-adic representation of $i-1$,
\[
i-1 =  C_{0} +C_{1} q^{1} + \cdots + C_{L-1} q^{L-1}\equiv
\sum\limits_{\nu=1}^L (j_{\nu}-1) q^{\nu-1}.
\]
For $d> 1$ the construction is similar \cite{KhQuant:09}.

Suppose that the quantics image for certain  $N$-$d$  tensor  
(i.e. an element of $D$-dimensional quantized tensor space with $D=d \log_q N= dL$)
can be effectively represented (approximated) in the low-rank canonical or TT format
living in the higher dimensional tensor space $\mathbb{Q}_{q,dL}$.
In this way we introduces the QTT approximation of an $N$-$d$ tensor. 
For given rank $\{r_k\}$, ($k=1,...,dL$) QTT approximation of an $N$-$d$ tensor  
the number of  representation parameters can be estimated by
$$
d q r^2 \log_q N \ll N^d, \quad \mbox{where}\quad r_k \leq r, k=1,...,dL,
$$ 
providing $\log$-volume scaling in the size of initial tensor $O(N^d)$.
The optimal choice of the base $q$ is shown to be $q=2$ or $q=3$ \cite{KhQuant:09}, however
the numerical realizations are usually implemented  by using binary coding, i.e. for $q=2$. 

The {\it principal question} arises: either there is the rigorous theoretical substantiation of 
the QTT approximation scheme that establishes it as the {\it new powerful approximation 
tool} applicable to the broad class of data, or this is simply the heuristic 
algebraic procedure that may be efficient on certain numerical examples.

The answer is positive: the power of QTT approximation method is due to the 
perfect rank-$r$ decomposition discovered in \cite{KhQuant_P:09} for 
the wide class of function-related tensors obtained by sampling a continuous functions
over uniform (or properly refined) grid:
\begin{itemize}
\item $r=1$ for complex exponents, 
\item $r=2$ for trigonometric functions and  for Chebyshev polynomials 
sampled on Chebyshev-Gauss-Lobatto grid,
\item $r\leq m+1$ for polynomials of degree $m$, 
\item $r$ is a small constant for some wavelet basis functions, etc.
\end{itemize}
all independently on the vector size $N$ and all applicable to the general case $q=2,3,...$.


Note that the name quantics (or quantized) tensor approximation (with a shorthand QTT) 
originally introduced in 2009 by \cite{KhQuant_P:09} is reminiscent of 
the entity ``quantum of information'', that mimics 
the minimal possible mode size ($q=2$ or $q=3$) of the quantized image.
Later on the QTT approximation method was also renamed  
{\it vector tensorization} \cite{Gras_Quant:10}. 

It is worth to comment that imposing different names for the same tensor
approximation method may lead to  
confusions in historical surveys on the topic.
For example, \S 14 in monograph \cite{Hack_Book12}, actually describing the quantized
tensor approximation method, 
designates this approach as the vector tensorization following 
the preprint \cite{Gras_Quant:10} (2010), 
but missing the reference on the original MPI MiS preprint \cite{KhQuant_P:09} (2009)
published merely one year earlier, where the method was introduced 
(and analyzed theoretically) as quantics-TT (QTT) approximation for vectors.
A recent survey paper \cite{Hack_SIAMSurv:14} presents yet another misleading interpretation 
by suggesting that the original name of the approach is the vector tensorization introduced 
in preprint \cite{Gras_Quant:10}, and stating that the quantics-TT (QTT) method has appeared 
later in 2011 in the journal version \cite{KhQuant:09}. 
Notice that the notion ``QTT tensor approximation'' has been since 2009 
recognized in the literature as the common notation (see, for example \cite{KhOsel_1_P:09})
for this efficient tensor approximation tool, now applied to various nontrivial 
multidimensional models.
This remark is an attempt to establish correct chronology in the development 
of the QTT approximation techniques. 

Concerning the matrix case, 
it was first found in \cite{Osel-TT-LOG:09} by numerical tests 
that in some cases the dyadic reshaping of an $N \times N$ matrix with $N=2^L$ may lead to a small
TT-rank of the resultant  matrix rearranged to the tensor form.
The efficient low-rank QTT representation for a class of discrete multidimensional operators
(matrices) was proven in \cite{KazKhor_1:10,KhOsel_3_DMRG:10}.
Moreover, based on the QTT approximation, the important algebraic matrix operations
like FFT, convolution and wavelet transforms can be performed in $O(\log_2 N)$ complexity
\cite{DoKhSav:11,KazKhTyrt_Toepl:11,KhMiao:13} (see also \S\ref{ssec:QTT-Oper} below).

The recently introduced combined two-level Tucker-TT \cite{KhQuant:09} 
and QTT-Tucker \cite{DoKh_QTTT:12} formats encapsulate the
benefits of the Tucker, TT and QTT representations and relax certain disadvantages 
of their independent use. 
The numerical experiments clearly indicate
that the two-level QTT-Tucker format outperforms both the TT and global QTT representations
applied independently due to systematic reduction of the effective tensor ranks.

To conclude this section we note that the remarkable property concerning the uniform 
in the grid size $N$ 
bound on the QTT-ranks for some classes of function related vectors (tensors) may address 
the natural question:
is there some meaningful interpretation of quantized tensor formats in the infinite 
dimensional setting ? 
The constructive answer on this question should take into account the following issues: 
(a) the target $N$-vector, $N=2^L$,  is likely to represent a  
sequence of continuous functions corresponding to different $L$  
so that one may expect the continuous limit as $L\to \infty$; 
(b) the QTT representation of interest is supposed to be used
as an intermediate  quantity involved in the solution of certain PDE, so that it would be
possible to calculate functionals and operators on such an element 
involved in the solution process designed in the quantized tensor space.
The latter is exactly the point why the QTT representation of the basic 
partial differential operators and functional transforms 
(say, discrete elliptic operators, FFT, wavelet, 
and circulant convolution) were also in the focus of the
QTT approximation theory (see also \S\ref{ssec:QTT-Oper}). 
The more detailed discussion on this intriguing topic will be addressed in the 
forthcoming papers.

%

\subsection{Functions in quantized tensor spaces}\label{ssec:QTT-Fun}

The simple isometric folding of a multi-index data array into the $2\times 2 \times ... \times 2$ 
format leaving in the virtual (higher) dimension $D=d \log N$ is 
the conventional reshaping operation in computer data representation. 
The most gainful features of numerical computations in the quantized tensor space 
appear via the remarkable rank-approximation properties discovered  
for the wide class of function-related vectors/tensors \cite{KhQuant:09}.

Next lemma presents the basic results on the rank-$1$ (resp. rank-$2$)
$q$-folding representation of the exponential (resp. trigonometric) vectors.

\begin{lemma}  \label{q-splitting} (\cite{KhQuant:09})
 For given $N=q^L$, with $q=2,3,...$,  $L\in \mathbb{N}$, and 
 $z  \in \mathbb{C}$, the exponential $N$-vector, 
$
Z :=\{x_n = z^{n-1}\}_{n=1}^N,
$
can be reshaped by the  $q$-folding 
to the rank-$1$, $q^{\otimes L}$-tensor, 
\begin{equation}\label{eq exp-vect q}
\mathcal{F}_{q,L}:  Z \mapsto  \textsf{\textbf{Z}}=
 \otimes_{p=1}^L [1 \; z^{q^{p-1}}\,...\,z^{(q-1)q^{p-1}} ]^T  \in \mathbb{Q}_{{q},L}.
\end{equation}
The number of representation parameters specifying the QTT image is reduced dramatically from $N$ to $q L $.

The trigonometric $N$-vector, 
$
 T= \Im m(Z) :=\{t_n =  \sin(\omega (n-1))\}_{n=1}^N,\quad \omega \in \mathbb{R},
$
can be reshaped by the successive $q$-adic folding 
\[
\mathcal{F}_{q,L}:   T \mapsto \textsf{\textbf{T}} \in \mathbb{Q}_{{q},L},
\]
to the $q^{\otimes L}$-tensor $\textsf{\textbf{T}}$, which has both the canonical 
$\mathbb{C}$-rank, and the TT-rank equal exactly $2$.
The number of representation parameters does not exceed $4 q L $. 
\end{lemma}

\begin{example} 
In the case  $q=2$, the single $\sin$-vector 
has the explicit rank-$2$ QTT-representation in $\{0,1\}^{\otimes L}$ 
(see \cite{DoKhSav:11,OselQTTfunc:10})
with $k_p=2^{p-L} i_p -1$, $i_p\in \{0,1\}$,
\begin{equation*}\label{eq sin-vect}
T \mapsto \textsf{\textbf{T}}=\Im m (\textsf{\textbf{Z}})= 
[\sin\, \omega k_1 \cos\,\omega k_1] \Join_{p=2}^{L-1}
\left[
\begin{array}{c c}  \cos\,\omega k_p &-\sin\,\omega k_p \\  
\sin\,\omega k_p &\cos\,\omega k_p \end{array}
\right] \Join \left[
\begin{array}{c} \cos\,\omega k_L \\ \sin\,\omega k_L \end{array}
\right].
\end{equation*}
\end{example}

Other results on QTT representation of polynomial, Chebyshev polynomial, Gaussian-type 
vectors, multi-variate polynomials and their piecewise continuous versions have been derived 
in \cite{KhQuant:09} and in
subsequent papers \cite{KhOsel_3_DMRG:10,Gras_Quant:10,DoKhOsel:11},
substantiating the capability of numerical calculus in quantized tensor spaces.

In computational practice the binary coding representation with $q=2$
is the most convenient choice, though the Euler number $q_\ast=\operatorname{e}\approx 2.7...$ 
is shown to be the optimal value \cite{KhQuant:09}).

The following example demonstrates the low-rank QTT approximation can be applied 
for $O(|\log \varepsilon|)$ complexity integration of functions. 
Given continuous function $f(x)$ and weight function $w(x)$, $x\in [0,A]$, 
consider the rectangular $N$-point quadrature, $N=2^L$, ensuring the error bound  
$|I-I_N|=O(2^{-\alpha L})$. Assume that the corresponding functional vectors 
allow low-rank QTT approximation. Then the rectangular quadrature can be implemented as   
the scalar product on QTT tensors, in {$O(\log N)$} operations.
$$
\int_{-1}^1 w(x) f(x) dx \approx I_N(f):=h\sum\limits_{i=1}^N w(x_i) f(x_i) =
\langle \textsf{\textbf{W}},\textsf{\textbf{F}} \rangle_{QTT},
\quad \textsf{\textbf{W}},\textsf{\textbf{F}} \in \bigotimes_{\ell=1}^L \mathbb{R}^2.
$$
Example \ref{exm:QTT-quadr} illustrates the uniform bound on the QTT rank
for nontrivial highly oscillating functions and with choice $w(x)=1$, see Figure \ref{fig:QTT-quadr}.
Here and in the following the threshold error like $\epsilon_{QTT}$ corresponds to the Euclidean norm.


\begin{figure}[tbh] 
\begin{center}
\includegraphics[width=7.5cm]{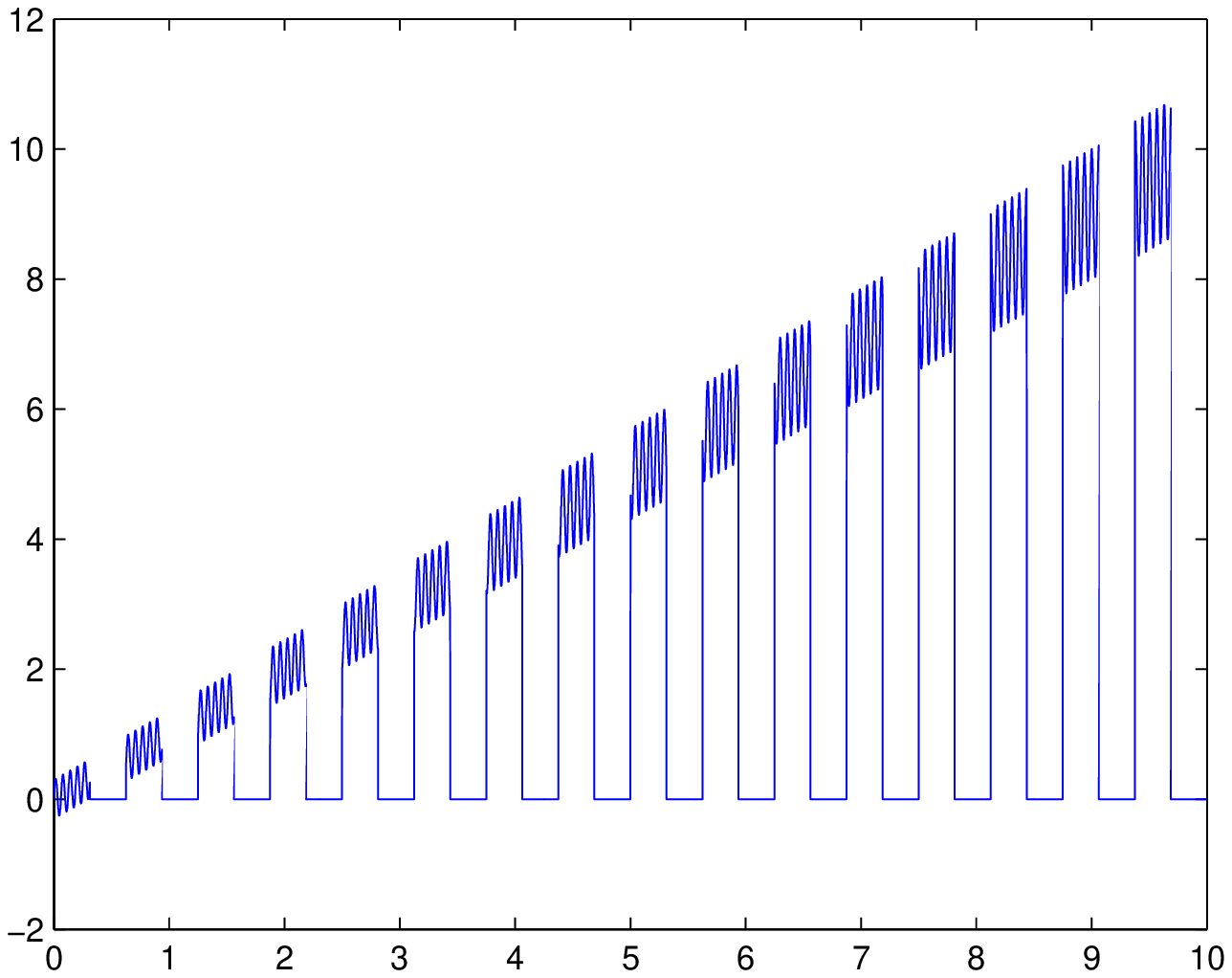} \quad 
\includegraphics[width=7.5cm]{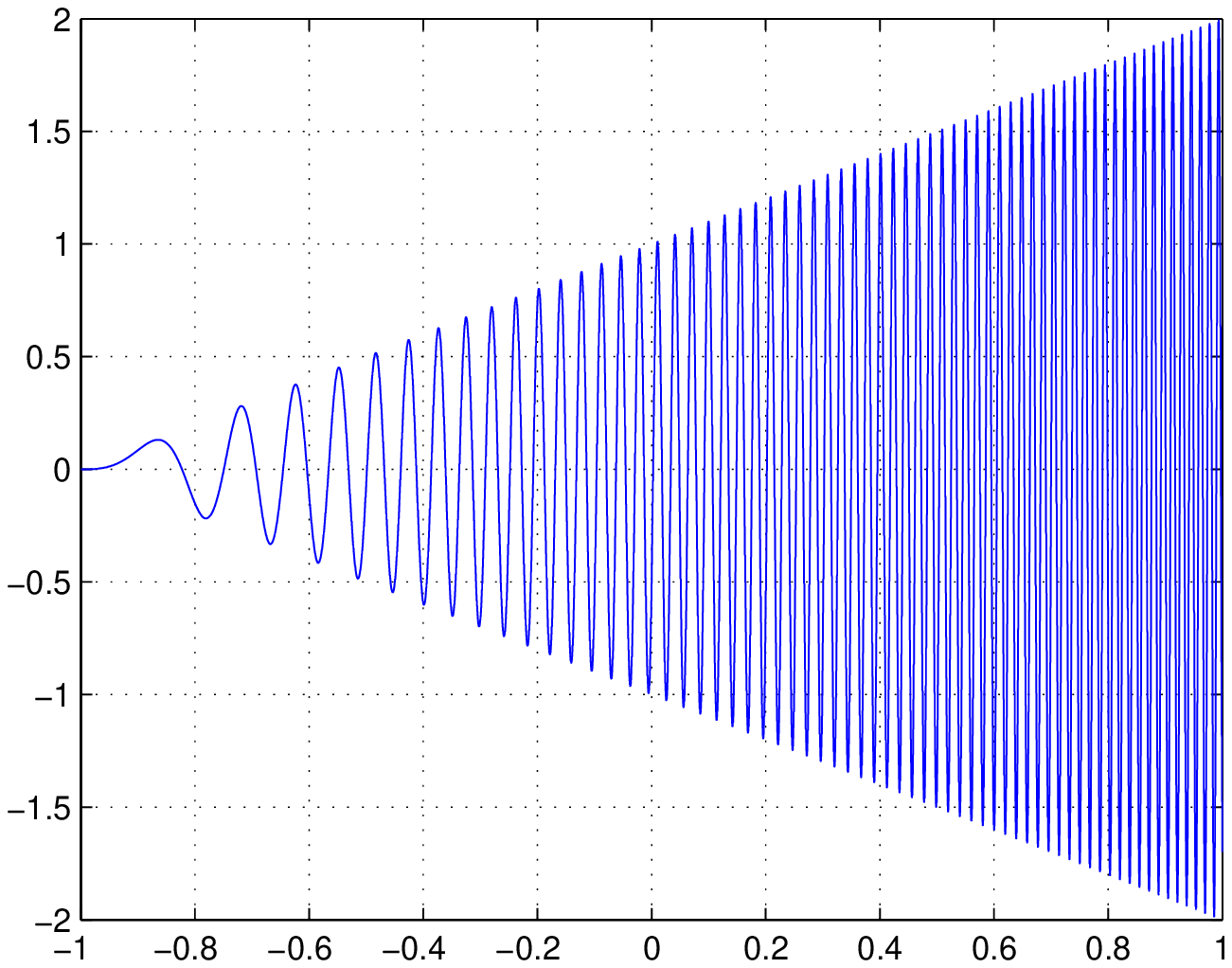}
\end{center}
\label{fig:QTT-quadr}
\caption{Visualizing functions $f_3$ (left) and $f_4$.}
\end{figure}

\begin{example} \label{exm:QTT-quadr}
Highly oscillated and singular functions on $[0,A]$, $\omega=100$, $\epsilon_{QTT}=10^{-6}$,
\begin{align*}
f_3(x) & =  \left\{ \begin{array}{ll}
       x+ a_k \sin(\omega x), & x\in 10\left(\frac{k-1}{p};\frac{k-0.5}{p}\right]\\
       0,                     & x\in 10\left(\frac{k-0.5}{p};\frac{k}{p}\right]
                    \end{array}
            \right. \\ \nonumber
f_4(x) & =  (x+1)\sin(\omega (x+1)^2), \quad x\in [0,1], 
\quad \mbox{(Fresnel integral)}. 
\end{align*}
where function $f_3(x)$, $x\in [0,10]$, $k=1,...,p$, $p=16$, $a_k=0.3 + 0.05(k-1)$,
is recognized on three different scales.
\end{example}
The average QTT rank over all mode ranks for the corresponding functional vectors are given in the next table.
The maximum rank over all the fibers is nearly the same as the average one.

\begin{table}[h]
\begin{center}
\begin{tabular}{c|c|c}
$N \setminus\overline{r}$  &  $r_{QTT}(f_3)$& $r_{QTT}(f_4)$ \\
\hline
$2^{14}$  &  3.5 & 6.5\\
\hline
$2^{15}$  &  3.6 & 7.0 \\
\hline
$2^{16}$  & 3.6  & 7.5 \\
\hline
$2^{17}$  & 3.6 & 7.9 
\end{tabular}
\end{center}
\caption{Average QTT ranks of $N$-vectors generated by $f_3$ and $f_4$.}
\end{table}

Notice that 1D and 2D numerical quadratures  based on interpolation by Chebyshev polynomials
have been developed \cite{HalTref:12}. Taking into account that Chebyshev polynomial sampled 
on Chebyshev grid has exact rank-$2$ QTT representation \cite{KhQuant:09}, allows us 
the efficient numerical integration by  Chebyshev interpolation using the QTT approximation.

\subsection{Operators in quantized tensor spaces}\label{ssec:QTT-Oper}

This section discusses the quantised representation of operators/matrices.
The explicit low-rank QTT representations for the wide class of discrete elliptic operators (FEM/FDM matrices)
was recently developed in 
\cite{KazKhor_1:10,KhQuant:09,KhOsel_3_DMRG:10,KazKhTyrt_Toepl:11,DoKh_CME_NLAA:13}.

As the first result is this direction the explicit rank-$3$ QTT representation of ${\Delta}_1$ 
was introduced \cite{KazKhor_1:10},
\begin{eqnarray}\label{Eq:Laplace-1-DD-QTT}
\Delta_1 & = &
\begin{bmatrix}
I       &\MTran{J}      &J      \\
\end{bmatrix}
 \Join
\begin{bmatrix}
I       &\MTran{J}      &J      \\
        &J      &               \\
        &       &\MTran{J}      \\
      \end{bmatrix}^{\otimes\left(d-2\right)}
    \Join  
\begin{bmatrix}
2I-J-\MTran{J}  \\
-J              \\
-\MTran{J}      \\
\end{bmatrix},\nonumber
\end{eqnarray}
with the Pauli matrices
\begin{equation*}\label{eq I, J}
I=
\left[\begin{array}{c c}   1 & 0 \\  
0 & 1 \end{array}\right], \quad
J=
\left[\begin{array}{c c}   0 & 1 \\  
0 & 0 \end{array}\right].
\end{equation*}
Other results concerning QTT representation of ${\Delta}_d$ and its inverse for $d\geq 2$ 
are collected in Proposition \ref{Df:TT-Rank_Laplac} below.

The analysis of the low QTT-rank approximations of elliptic operator inverse
for $d\geq 2$ is based on certain assumptions on the QTT-rank of the
 matrix exponential family.
\begin{conjecture}\label{Exp-matr rank-r TT}
For any given $\varepsilon>0$, and for fixed $a,b>0$, 
the family of matrix exponentials, 
$\{\exp(-t_{k}{{\Delta}_1})\}$, $t_{k}>0 $, $k=-M,...,M$, 
allows the QTT $\varepsilon$-approximation with rank-$r_\Delta$
being uniformly bounded in the grid size $N$
and in the scaling factors $t_{k} \in [a,b]\subset \mathbb{R}_{>0}$  
(see Table \ref{2-matrix ranks} for numerical justification).
\end{conjecture}
Table \ref{2-matrix ranks} represents the average QTT-ranks in approximation  
of certain function related matrices up to fixed tolerance
$\varepsilon_{QTT}=10^{-5}$.
It includes the important example of matrix exponential
(cf. Conjecture \ref{Exp-matr rank-r TT}).
The matrix $A=diag({f(\|x\|^2)})$, ($x=(x_1,x_2)$, $|x_i|\leq 1$) is a diagonal matrix with
diagonal entries obtained by sampling a function $f(\|x\|^2)$ over uniform grid points situated 
on the diagonal $x_1=x_2$.

We observe that rank parameters depend very mildly on the grid size.
\begin{table}[h]
\begin{center}
\begin{tabular}{|c|c|c|c|c|}
\hline
$N \setminus \overline{r}$  & $e^{-\alpha \Delta_1}, \alpha=0.1/1/10/10^2$ &   
$diag({1/ \|x\|^2})$ & $diag({e^{-\|x\|^2}})$ \\
\hline
$2^9$ & 6.2/6.8/9.7/11.2 &  5.1 & 4.0\\
\hline
$2^{10}$ & 6.3/6.8/9.5/10.8 &  5.3 & 4.0\\
\hline
$2^{11}$ & 6.4/6.8/9.0/10.4  &  5.5  & 4.1 \\
\hline
\end{tabular}
\end{center}
\caption{Average QTT ranks of $N\times N$-matrices for large  $N=2^p$.}
\label{2-matrix ranks}
\end{table}
We note that the QTT ranks of the matrix $A=diag({f(\|x\|^2)})$ coincides with those for  
the generating vector $X$ that follows from the explicit QTT representations 
(see Def. \ref{def:OTT_Matformat}).
Given vector $X$ and the corresponding rank-${\bf r}_X$ QTT-tensor 
$\textsf{\textbf{X}}=\textsf{X}^{(1)}\Join \textsf{X}^{(2)}\Join \cdots 
\Join \textsf{X}^{(d)}\in \mathbb{X}$, then the QTT representation
of the matrix $A=diag(X)$ is given by
 \[
{\bm{\mathcal{A}}}
= {\cal A}^{(1)}\Join {\cal A}^{(2)}\Join \cdots \Join {\cal A}^{(d)}, 
\quad {\cal A}^{(2)} = diag( \textsf{X}^{(k)}),
\]
 where $diag(\textsf{X}^{(k)})$ is the matrix diagonalization of the core tensor $\textsf{X}^{(k)}$.

Define the anisotropic $d$-Laplacian
\begin{equation}\label{anis dLapl}
\Delta_{d,\alpha}:= \sum\limits_{\ell=1}^d \alpha_\ell
\bigotimes_{k=1}^d \Delta_1^{\delta_{\ell-k}},\quad \alpha_\ell>0, \quad \delta_{m} \quad
\mbox{is the Kronecker symbol}.
\end{equation}
Now we can prove the following Lemma (see \cite{KhQuant:09}).
\begin{lemma}\label{Anis Lapl Inv} 
Under claims of Conjecture \ref{Exp-matr rank-r TT} on the QTT-rank of the univariate
matrix exponentials, 
there exist $C,\beta >0$ independent of $d$, such that for all $M\in \mathbb{N}$,
\begin{equation}\label{appr_dLaplm1}
\|{\Delta}_{d,\alpha}^{-1} -{B}_M \|\leq C e^{- \beta \sqrt{M}}, \quad \beta >0,
\end{equation}
where ${B}_M$ is defined as
\begin{equation}\label{Lapl Inv}
  {B}_M:=\sum\limits_{k=-M}^{M}c_{k}\bigotimes_{\ell=1}^{d}
  \exp(-t_{k}{\alpha_\ell {\Delta}_1}),\quad
t_{k}=e^{k\mathfrak{h}},\quad c_{k}=\mathfrak{h}t_{k},\quad \mathfrak{h}=\pi/\sqrt{M}. 
\end{equation}
\end{lemma}
Lemma \ref{Anis Lapl Inv} implies that the matrix family $\{B_M\}$ possesses the low-rank approximation  
(or preconditioner if $M$ is small) to the anisotropic $d$-Laplacian inverse ${\Delta}_{d,\alpha}^{-1}$,
whose ranks scale like $O(\log^2 \varepsilon)$, where  $\varepsilon$ is the error threshold.
The constant $\beta>0$ depends logarithmically on the grid size $N$, while $C$ scales linearly 
in the norm of inverse matrix $\|{\Delta}_{d,\alpha}^{-1}\|$.

The following statement summarizes the previous discussion and the related results in \cite{KazKhor_1:10}.
\begin{proposition}\label{Df:TT-Rank_Laplac}
For $ d\geq 2$ the canonical, TT and QTT rank estimates hold:
\begin{align*} 
rank_C({\Delta}_d)  & = d, \quad rank_{TT}({\Delta}_d)=2, \quad rank_{QTT}({\Delta}_d) = 4. 
\\ \nonumber
rank_{QTT}({\Delta}_1) & = 3, \quad rank_{QTT}({\Delta}_1^{-1}) \leq 5. 
\end{align*}
Given $\varepsilon>0$, then for the $\varepsilon$-rank we have
\begin{align} \label{QTT_LaplInv}
rank_{TT}({\Delta}_d^{-1}) & \leq rank_C({\Delta}_d^{-1}) \leq C
| \log \varepsilon | \log N. 
\end{align}
Moreover, under claims of Conjecture \ref{Exp-matr rank-r TT} there holds 
\begin{equation}\label{QTT_LaplInv}
 rank_{QTT}({\Delta}_d^{-1})  \leq C | \log \varepsilon |^2 \log N.
\end{equation}
In both cases the constant $C$ does not depend on $d$.
\end{proposition}


The $d$-dimensional FFT over $N^{\otimes d}$ grid can be realized
on the rank-$k$ tensor  
with the linear-logarithmic cost $O(d k N\log_2 N)$,  due to the rank-$1$
factorized representation
\[
{\cal F}_{N}^{(d)}=(F_{N}^{(1)}\otimes I ...\otimes I)
(I\otimes F_{N}^{(2)}...\otimes I ) ... (I\otimes I ...
\otimes F_{N}^{(d)})\equiv F_{N}^{(1)}\otimes ... \otimes F_{N}^{(d)},
\]
where $ F_{N}^{(\ell)}\in \mathbb{R}^{N \times N} $ represents
the univariate FFT matrix along mode $\ell$.

The super-fast Fourier  transform of $N$-$d$ tensors can be computed in $\log$-volume complexity, 
$O(d \log^2 N)$, by using the QTT approximation as proposed in \cite{DoKhSav:11}.

Direct convolution transform of $N$-$d$ tensors represented in the canonical/Tucker  formats
\cite{Khor1:08,KhKh3:08}, see (\ref{eqn:canon_conv}),
can be implemented in $O(d \log N)$ operations by representing canonical/Tucker vectors
by the low-rank QTT tensors (canonical-QTT or Tucker-QTT formats) 
\cite{VeKhSchn_QTTHF:11,KhVBSchn_TEI:12}.

The super-fast convolution transform of the complexity $O(d \log^2 N)$
using the explicit QTT representation of multilevel Toeplitz matrices is developed in 
\cite{KazKhTyrt_Toepl:11}. 

The super-fast QTT wavelet transform of logarithmic complexity by the exact low-rank
representation of the wavelet transform matrix was described in \cite{KhMiao:13}
(see also \cite{OselTyr_WT:11} for the related discussion).  

\subsection{Multilinear algebra and tensor-rank truncation}\label{ssec:MLA_TTr}

Low-parametric tensor formats provide  prerequisites for 
multidimensional algebraic calculus since the storage complexity of rank-structured tensors
scales linearly in dimension $d$. However, the numerical capability of $\varepsilon$-truncated 
tensor operations is essentially determined by the following issues:
\begin{itemize}
\item Understanding of nonlinear tensor approximation in the canonical, Tucker, and MPS/TT formats
\item Developments on efficient multilinear algebra and rank optimization algorithms
\item Approximation theory for functions and operators in (quantized) tensor spaces
\item Construction of stable iterative methods for solving multidimensional PDEs in tensor formats.
\end{itemize}

It is worth to note that the important multilinear algebraic
operations with canonical, Tucker and TT tensors can be  
implemented with linear complexity scaling in the univariate mode size $N$ and in the dimension $d$ 
by representing them in terms of tensor operations on rank-$1$ elements. 

For example, following \cite{KhBoVe1:06}, for the rank-$R_1$ and rank-$R_2$ canonical tensors 
$\textsf{\textbf{X}},\textsf{\textbf{Y}}\in\mathbb{R}^{\cal I}$, ${\cal I}:=\{1,...,N\}^d$, we have 
$$
\langle \textsf{\textbf{X}},\textsf{\textbf{Y}} \rangle =
\sum\limits_{k=1}^{R_1} \sum\limits_{m=1}^{R_2}
\prod\limits_{\ell=1}^d \left\langle  \textsf{X}^{(\ell)}_k,\textsf{Y}^{(\ell)}_m \right\rangle,
$$
while the Hadamard product is computed by 
\begin{equation*}\label{eqn:TIP_Had}
\textsf{\textbf{X}} \odot \textsf{\textbf{Y}}:=
\sum\limits_{k=1}^{R_1} \sum\limits_{m=1}^{R_2}
\left( \textsf{X}^{(1)}_k \odot  \textsf{Y}^{(1)}_m  \right)
\otimes \ldots \otimes
\left(  \textsf{X}^{(d)}_k \odot   \textsf{Y}^{(d)}_m \right).
\end{equation*}

We define the {\it discrete convolution product} of two convolving  tensors \cite{KhBoVe1:06} by
\[
\textsf{\textbf{X}} \ast \textsf{\textbf{Y}}:=\left[\sum\limits_{{\bf k} \in 
    {\cal I}}{\textsf{X}}_{\bf k}\textsf{Y}_{{\bf j}-{\bf k}}\right]_{{\bf j}
  \in{\cal J}}, \quad {\cal J}:=\{1,...,2N-1\}^d.
\]
The convolution product of two tensors in the canonical format
can be realized in $O(d R_1 R_2  N \log N )$ operations \cite{KhBoVe1:06} relying on the representation
\begin{equation}\label{eqn:canon_conv}
\textsf{\textbf{X}} \ast  \textsf{\textbf{Y}}  :=
\sum\limits_{k=1}^{R_1} \sum\limits_{m=1}^{R_2}
\left(\textsf{X}^{(1)}_k \ast \textsf{Y}^{(1)}_m  \right) \otimes \ldots \otimes
\left(  \textsf{X}^{(d)}_k \ast   \textsf{Y}^{(d)}_m \right)\in \mathbb{R}^{\cal J},
\end{equation}
where $\textsf{X}^{(\ell)}_k \ast \textsf{Y}^{(\ell)}_m  \in \mathbb{R}^{2N-1}$ denotes the convolution 
product of $N$-vectors defined as
\[
\textsf{X}^{(\ell)}_{k}\ast \textsf{Y}^{(\ell)}_{m}=
\left[\sum\limits_{n=1}^n x_n y_{j-n}\right]_j, \quad j=1,...,2N-1.
\]
Hence, the one-dimensional convolution can be performed by FFT in
$O(N\log N)$ operations. Similarly, the above mentioned tensor operations for Tucker tensors
can be reduced to 1D linear algebra \cite{KhBoVe1:06,Khor1:08}.

The formatted implementation of the scalar product of TT tensors \cite{Osel-TT:09}
$$
\langle \textsf{\textbf{X}},\textsf{\textbf{Y}} \rangle =
\langle \textsf{\textbf{X}} \odot \textsf{\textbf{Y}} ,{\bf 1} \rangle
$$
can be performed by using the Hadamard product in TT format
$
\textsf{\textbf{Z}}= \textsf{\textbf{X}} \odot \textsf{\textbf{Y}}:\;  
\textsf{Z}^{(k)}_{i_k}=\textsf{X}^{(k)}_{i_k}\otimes \textsf{Y}^{(k)}_{i_k}.
$

Furthermore, the standard multilinear operations like addition, the scalar, Hadamard, 
contracted and convolution products, etc. in the QTT, QTT-canonical or QTT-Tucker representations
can be implemented in $O(d \log N)$ operations and storage costs.
This allows fast computations on large spatial grids without practical limitations on $N$, 
where $N$ is usually associated with the univariate grid size.

Representation of tensors in low separation rank formats is the key
point in the design of fast tensor-structured numerical methods for large-scale higher
dimensional simulations. In fact, it allows the implementation of basic linear and
bilinear algebraic operations on tensors mentioned above
with linear complexity in the univariate tensor size
(see \cite{Khor1:08,KhorSurv:10,KazKhor_1:10,VeKh_Diss:10,DoKhSav:11}).
However, the bilinear tensor operations normally increase the
separation rank of the resultant tensor.

To perform  computation over  nonlinear set (manifold) of rank-structured tensors ${\cal S}$
(say, the canonical, Tucker, MPS/TT, QTT, and QTT-Tucker formats) with controllable complexity,
we need to perform a ``projection'' of the current iterand in ${\cal S}_0 \supset {\cal S}$ onto 
that manifold ${\cal S}$ by using the ``formatted`` tensor operations. This action
  is fulfilled by implementing the tensor truncation operator
  ${T_{\cal S}}:\mathbb{W}_{{\bf n},d} \to {\cal S}$ defined by
\begin{equation}\label{Cost f1}
 \textsf{\textbf{A}}_0\in {\cal S}_0\subset \mathbb{W}_{{\bf n},d} :
 \quad {T}_{\cal S} \textsf{\textbf{A}}_0=  \operatorname{argmin}_{ \textsf{\textbf{T}}\in {\cal S} }
 \|\textsf{\textbf{A}}_0 - \textsf{\textbf{T}}\|,
\end{equation}
that reduces to a challenging nonlinear approximation problem.
The replacement of $\textsf{\textbf{A}}_0$ by its approximation in ${\cal S}$
is called the {\it tensor truncation} to ${\cal S}$ and denoted
by ${T}_{\cal S} \textsf{\textbf{A}}_0 $. 
In practice, the computation of the minimizer
${T}_{\cal S} \textsf{\textbf{A}}_0$ can be performed only approximately.
The set ${\cal S}$ of rank-structured tensors can be chosen adaptively in order to
control the approximation error $\varepsilon>0$,
\[
 \|\textsf{\textbf{A}}_0 - {T}_{\cal S} \textsf{\textbf{A}}_0\|\leq \varepsilon.
\]
 
In the case of Tucker, TT/QTT and QTT-Tucker
formats the quasi-optimal approximation can be computed by conventional QR/SVD algorithm 
\cite{KoldaB:07,OsTy_TT:09,GrasHTUCK:09,DoKh_QTTT:12},
also known in the physical literature as the Schmidt decomposition.
In particular, the Tucker tensors can be approximated by the so-called higher order SVD (HOSVD),
\cite{DMV-SIAM:00}.
Robust SVD-based algorithms are applicable since the Tucker and TT ranks can be controlled 
by a certain matrix rank.
Indeed, for MPS/TT format we have the equivalent definition in terms of 
the TT-unfolding matrix,
$TT[{\bf r}]:=\{\textsf{\textbf{A}}\in \mathbb{V}_{\bf n}: \rank 
A_{[p]}\leq r_p$, p=1,...,d-1 \}, where the TT-unfolding matrix $A_{[p]}$ is defined by
\begin{equation*}\nonumber
A_{[p]}:= \textsf{\textbf{A}}
(\underbrace{{j_1 \: j_2\: \ldots j_p}}_{\mbox{column ind.}}; 
\underbrace{\textcolor{blue}{j_{p+1}\ldots j_d }}_{\mbox{row ind.}}).
\end{equation*}
For the Tucker format we have the alternative definition
$
{\cal T}[{\bf r}]:=\{ {\textsf{\textbf{A}}} \in \mathbb{V}_{\bf n}: \mbox{rank}\, 
A_{(p)} \leq r_p, p=1,...,d \} $, where the Tucker unfolding matrix $A_{(p)}$ is defined by
\begin{equation*}\nonumber
A_{(p)}:= \textsf{\textbf{A}}
(\underbrace{\textcolor{blue}{j_1 \: j_2\: \ldots j_{p-1}}}_{\mbox{row ind.}}; 
\underbrace{{j_p}}_{\mbox{column ind.}};
\underbrace{\textcolor{blue}{j_{p+1}\ldots j_d }}_{\mbox{row ind.}}).
\end{equation*}

In turn, the canonical and TC ranks can not be identified with certain matrix ranks
that may result in instabilities of approximation schemes.
Approximation in the $R$-term canonical format is considered as a subtle problem that cannot 
in general be solved by  a stable algorithms in polynomial cost. 
One of the reasons is that the set of canonical tensors with rank $\leq R$ is not closed as shown by the following
example.
\begin{example} \label{exm:can_Nclosed}
The tensor $\textsf{\textbf{F}}$, with $rank_{Can}(\textsf{\textbf{F}})={d}$,
generated by sampling of the function $f(x)=x_1+...+x_d$ on a tensor grid
can be approximated with arbitrary precision by rank-$2$ elements
$$
\textsf{\textbf{F}}= \lim_{\varepsilon\to 0}\frac{{\bigotimes}_{\ell=1}^d(1+ \varepsilon X_\ell)-1}{\varepsilon}.
$$ 
\end{example}
However, for some classes of functions and operators the robust analytic 
methods based on $R$-term explicit $\operatorname{sinc}$-quadrature approximations can be 
successfully applied 
(see \cite{KhorSurv:10} and references therein). 

Notice that the so-called reduced HOSVD (RHOSVD) method introduced in \cite{KhKh3:08}
provides robust rank reduction for the canonical tensors with large initial rank.
This is  a two-step method which, first, reduces the canonical tensor to low-rank Tucker
form by SVD-based low-rank approximation of the factor matrices,
and then applies the ALS type rank reduction iteration to the small-size Tucker core to represent it 
in canonical format.
The RHOSVD method was successfully applied  in electronic structure calculations
 \cite{KhKh3:08,VKh_Ex:09,VeKh_Diss:10,KhKhFl_Hart:09}.

\section{Tensor numerical methods for $d$-dimensional PDEs}\label{sec:QTT-solvers}

In this section, we discuss the benefits of tensor numerical methods (TNM) in two 
important applications: the nonlinear Hartree-Fock (HF) equation in electronic structure 
calculations (\S\ref{ssec:HaFockeqn})
and the real-time parabolic equations in many-particle dynamics (\S\ref{ssec:dynamics}).

The numerical challenge in the HF calculations is due to the presence of large number of
complicated convolution-type integrals in $\mathbb{R}^3$, multiply recomputed 
within the iterations on nonlinearity, as well as hard complexity scaling of the 
traditional methods with respect to the size of a molecular system
(see for example \cite{LeBris1:05,SaadRev:10} and references therein).
At this point the beneficial features
of tensor computations has been successfully realized in the form of fast 
grid-based black-box HF-solver that manifests linear 
complexity in the univariate grid size 
\cite{KhKh3:08,KhKhFl_Hart:09,VeKh_Diss:10,VKH_solver:13,KhVBSchn_TEI:12}.
The latter can be reduced to the logarithmic scale by implementation in quantized tensor spaces
\cite{VeKhSchn_QTTHF:11}.
Further steps toward TNMs for post Hartree-Fock calculations \cite{VeKhorMP2:13},
and for solving the Hartree-Fock problem for large lattice-structured and periodic systems 
\cite{VeBoKh:Ewald:14,VeKhorCorePeriod:14} are the subject of current research.

The real space numerical methods combined with FEM approximations 
have become attractive in electronic structure calculations as the possible alternative
to the traditional approaches 
(see \cite{LaPySu:1986,HaFaYaBeyl:04,TenAppInTuk:07,BiVale:11,Frediani:13,RahOsel:13} 
and references therein).

For a class of multi-dimensional parabolic problems including, in particular,
the molecular Schr\"odinger, the Fokker-Planck and master equations, 
the computational challenges arise due to the curse of dimensionality. 
We discuss the main issues which allow to understand how the $d$-dimensional dynamics
can be simulated in quantized tensor spaces with the linear complexity in $d$
and in log-log complexity in the mesh size for simultaneous space-time discretization. 
Further details can found in \cite{KhorSurv:10,DoKhOsel:11,GavlKh_Caley:11,DoKh_CME_NLAA:13}.
Exposition of tensor methods based on the Dirack-Frenkel variational principle implemented on the 
Tucker and MPS/TT tensor manifolds can be found in \cite{LubichBook:08,lubich-dynht-2012,okhs-dyntt-2012} 
and references therein. 

Recently, the  TNM were shown to be efficient for solving parameter dependent PDEs in the case of
high dimensional parametric space \cite{KhSchw1:09,KhOsel_2_SPDE:10,benner-spde-2013,DolKhLitMat:14}. 

Another popular topic in numerical analysis of
multidimensional PDEs is concerned with the so-called greedy methods and their applications 
\cite{BrisMaday-greedy_pdes-2009,Binev-conv_rate_greedy-2011,Cances-greedy-2011,Sueli-greedy_FP-2011}, 
which became attractive since the simplicity of greedy algorithms: 
the main step usually includes either rank-$1$ corrections low-rank tensors in other
formats commonly used in practice.
Due to the page limits, these topics will not be further addressed in this paper.

\subsection{Hartree-Fock equation in electronic structure calculations}\label{ssec:HaFockeqn}

\subsubsection{Problem setting}\label{sssec:HFE_Probl}

The HF equation for determination
of the ground state energy of a  molecular system consisting of $M$ nuclei
and $N_{orb}$ electrons  (closed shell case) is given by the following 
nonlinear eigenvalue problem in $H^1(\mathbb{R}^3)$ \cite{LeBris1:05},
\begin{equation} \label{eqn:HaFockProbl}
  ({\cal F}  \phi_i )({x}) = \lambda_i \, \phi_i({{x}}),\quad
  \int_{\mathbb{R}^3}\phi_i({x})\phi_j({x}) \, \mathrm{d} {x}=\delta_{ij},
  \; i,j=1,...,N_{orb},
\end{equation}
where  the non-linear integral-differential Fock operator ${\cal F}$ is given by
\begin{equation} \label{eqn:FockOper}
{\cal F} :=-\frac{1}{2} \Delta -V_c  + V_H + {\cal V}_E  ,\quad 
V_c =\sum_{\nu =1}^{M} \frac{Z_{\nu}}{\|{x}-{A}_{\nu}\|},
\end{equation}
with the Hartree potential, $V_H(x)$,  and the nonlocal exchange operator, ${\cal V}_E$,
defined by
$$
V_H({x}):=  \int_{\mathbb{R}^3}
\frac{\tau({y},{y})}{\|{x} - {y}\|} \,\mathrm{d} \textbf{y},\quad \mbox{and} \quad
{\cal V}_E\phi:= - \frac{1}{2}\int_{\mathbb{R}^3}
\frac{\tau({x},{y})}{\|{x} - {y}\|} 
\phi({y})\,\mathrm{d} {y},
$$
respectively.
Here, $1/{\|  \cdot \|}: \mathbb{R}^3 \to  \mathbb{R}$ corresponds to
the Newton potential, and $Z_{\nu} \in \mathbb{R}_{+}$,
${A}_{\nu} \in \mathbb{R}^3$ ($\nu=1,...,M$) specify charges and positions
of $M$ nuclei. The electron density matrix
$\tau:\mathbb{R}^3\times\mathbb{R}^3  \to\mathbb{R} $,
is given by
$
\tau({x},{y}) =  2 \sum_{i=1}^{N_{orb}} \phi_i({x})\phi_i^*({y}),
$
specifying the electron density $\rho(x)=\tau({x},{x})$.


Notice that the Hartree-Fock equation is a nonlinear eigenvalue problem in a sense, that
one should solve it when the nonlinear part $V_H + {\cal V}_E$
of the governing operator  depending on the eigenvectors is unknown. 
This dependence is expressed by the 3D convolution transform with the Newton convolving 
kernel, while the electron density $\rho(x)$ contains multiple 
strong singularities corresponding to each nuclei location. 
Therefore solution of the  HF equation requires iterative solvers, with multiply repeated
recalculation of these convolution integrals.

Usually, the Hartree-Fock equation is approximated by
the standard Galerkin projection of the initial problem
(\ref{eqn:HaFockProbl}) posed in $H^1(\mathbb{R}^3)$ (see \cite{LeBris1:05} for more details).
For a given finite Galerkin basis set $\{g_\mu \}_{1\leq \mu \leq N_b}$,
$g_\mu\in H^1(\mathbb{R}^3) $, the
molecular orbitals $\phi_i$ are expanded (approximately) by
 \begin{equation}\label{expand}
\phi_i=\sum\limits_{\mu=1}^{N_b} C_{\mu i} g_\mu, \quad i=1,...,N_{orb},
\end{equation}
yielding the Galerkin system of nonlinear equations
for the coefficients matrix $C=\{c_{\mu i} \}\in \mathbb{R}^{N_b \times N_{orb} }$,
concatenating the eigenvectors $C_i \in \mathbb{R}^{N_b}$
(and the density matrix $D=2CC^\ast \in \mathbb{R}^{N_b\times N_b}$)
\begin{equation}\label{SCF_EVP}
  F(D) C = SC\Lambda, \quad \Lambda= diag(\lambda_1,...,\lambda_{N_b}), \quad
  C^T SC =  I_{N_b},
\end{equation}
where $S$ is the overlap (stiffness) matrix for $\{g_\mu \}_{1\leq \mu \leq N_b}$,
and the Fock matrix
\begin{equation}
\label{G_Fock}
 F(D)= H + J(D) + K(D),
\end{equation}
is a sum of 
the stiffness matrix $H=\{h_{\mu \nu}\}$ of the core Hamiltonian
${\cal H}=-\frac{1}{2} \Delta + V_c $ (the single-electron integrals),
\[
h_{\mu \nu}= \frac{1}{2} \int_{\mathbb{R}^3}\nabla g_\mu \cdot \nabla g_\nu dx +
\int_{\mathbb{R}^3} V_c(x) g_\mu g_\nu dx, \quad 1\leq \mu, \nu \leq N_b,
\]
and the two nonlinear terms
$J(D) + K(D)$,
representing the Galerkin approximation to the Hartree and
exchange operators. This is the main computational task, which is traditionally calculated  by using the 
two-electron integrals tensor $\textsf{\textbf{B}}=[b_{\mu \nu \kappa \lambda}]$,
defined as follows:
Given the finite basis set $\{g_\mu \}_{1\leq \mu \leq N_b}$, $g_\mu\in H^1(\mathbb{R}^3) $, the
associated fourth order two-electron integrals (TEI) tensor,
$\textsf{\textbf{B}}=[b_{\mu \nu \kappa \lambda}]$, is defined entrywise by
\begin{equation} \label{btensor}
b_{\mu \nu \kappa \lambda}= \int_{\mathbb{R}^3}\int_{\mathbb{R}^3}
\frac{g_\mu(x) g_\nu(x) g_\kappa(y)g_\lambda(y) }{\| x-y \|} dx dy, \quad
\mu, \nu, \kappa, \lambda \in \{1,...,N_b\}.
\end{equation}
In the straightforward implementation based on the analytically precomputed integrals,
the computational and storage complexity for the TEI tensor is of order $O(N_b^4)$, or even $O(N_b^5)$,
that becomes non-tractable already for $N_b$ of order of several hundred.

Given TEI tensor, the $N_b\times N_b$ matrices $J(D)$ and $K(D)$ can be represented by
\begin{equation}\label{C-K matr}
J(D)_{\mu \nu}= \sum\limits_{\kappa, \lambda=1}^{N_b}
b_{\mu \nu, \kappa \lambda}D_{\kappa \lambda},\quad
K(D)_{\mu \nu}= - \frac{1}{2} \sum\limits_{\kappa, \lambda=1}^{N_b}
b_{\mu \lambda, \nu \kappa}D_{\kappa \lambda},
\end{equation}
with the low-rank symmetric density matrix, $D= 2 C C^T \in \mathbb{R}^{N_b\times N_b}$,
such that 
$
rank(D) =N_{orb}\ll N_b.
$
Equations (\ref{C-K matr}) can be rewritten in terms of the TEI matrix 
$B=[b_{\mu \nu, \kappa \lambda}]\in \mathbb{R}^{N_b^2 \times N_b^2}$.

The total energy is computed as
\[
E_{HF}= 2 \sum\limits_{i=1}^{N_{orb}} \lambda_i - \sum\limits_{i=1}^{N_{orb}} \left( \widetilde{J}_{i}
- \widetilde{K}_{i} \right),
\]
where $\widetilde{J}_{i}=(\phi_i,V_H \phi_i)_{L^2}= \langle C_i, J C_i\rangle$ and
$\widetilde{K}_{i}=(\phi_i, K \phi_i)_{L^2} = \langle C_i, K C_i\rangle$,  $i=1,...,N_{orb}$,
are the  Coulomb and exchange integrals in the basis of orbitals $\phi_i $. 
The resulting ground state energy of a molecule, $E_0$, for the given geometry
of nuclei, includes the nuclei repulsion energy $E_{nuc}$,
\begin{equation}
\label{eqn:ener_E0}
 E_0= E_{HF} + E_{nuc},\quad \mbox{ where } \quad 
E_{nuc}= \sum^M_{k=1} \sum^M_{m <k} \frac{Z_k Z_m}{\|x_k- x_m\|}.
\end{equation}
The standard quantum chemical implementations are based on the analytically 
precomputed set of the  two-electron integrals (\ref{btensor}) 
in a naturally separable Gaussian basis 
 with the computational and storage complexity for the TEI tensor  
of order $O(N_b^4)$, or even $O(N_b^5)$,
that becomes non-tractable already for $N_b$ of order of several hundred.


\subsubsection{Grid-based tensor numerical methods}\label{sssec:HF_TensorMeth}

The tensor-structured  numerical methods, both the name and the concept,
appeared during the work on the 3D grid-based tensor approach to the solution of the 
Hartree-Fock equation.
They lead  to ``black-box'' numerical treatment of the Hartree-Fock 
problem based on the  low-rank representation of the basis functions 
in a volume box,  using the $n\times n\times n$
$3D$  Cartesian grid  positioned arbitrarily with respect to the atomic centers 
\cite{KhKhFl_Hart:09,KhorVBAndrae:12,KhVBSchn_TEI:12,VeKh_Diss:10}.
In 2008 \cite{KhKh3:08,VKh_Ex:09} it was shown that within the tensor-structured paradigm
the core Hamiltonian and the $3D$ convolutions with the Newton kernel, 
involved in the Coulomb and exchange operators, can calculated in $1D$ complexity
by the rank-structured tensor operations reduced to $1D$ convolutions, Hadamard and scalar
products \cite{Khor1:08,KhKh3:08,VKH_solver:13}. 
Due to  elimination  of the analytical integrability requirements 
it gives a choice to use rather general physically relevant basis sets 
represented on the grid. High accuracy is achieved due to 
grid-sizes up to the order of $n \simeq 10^6$, yielding a volume box of size $n^3 \simeq 10^{18}$. 
It corresponds to mesh  resolution up to $h = 10^{-5}$ $\stackrel{\circ}{A}$ 
(close to sizes of atomic radii)
in the volume box with the equal sizes of $ 20\, \stackrel{\circ}{A}$ 
for each spatial variable. 
Matlab on a Sun station is used for all algorithms, without parallelizations and supercomputing.

Chronologically, two different approaches have been developed
for the 3D grid-based tensor-structured solution of the HF equation. Both use the rank-structured 
calculation of the core Hamiltonian \cite{KhorVBAndrae:12,VKH_solver:13} for a given grid-based
basis set.
\begin{itemize}
 
\item The tensor solver I does not use the two-electron integrals. Instead, the Coulomb and exchange
operators are recomputed ``on the fly'' using the refined 3D grids and rank-structured (1D) operations. 
\cite{VKh_Ex:09,KhKhFl_Hart:09,VeKh_Diss:10}. This approach has low storage demand, but might be time
consuming. It can be applicable to the Kohn-Sham type models.

\item The  ``black-box`` solver II based on calculation of the TEI matrix $B$ by the truncated 
Cholesky decomposition 
and the redundancy-free factorization by the algebraically reduced product basis
yielding the reduced storage consumption $O(N_b^3)$ \cite{VeKhorMP2:13,VeKhSchn_QTTHF:11,VKH_solver:13}. 
Its performance in time and accuracy is compatible with the benchmark packages 
in quantum chemistry based on analytical pre-calculation of involved multidimensional integrals.
\end{itemize}

In the following, we briefly discuss the tensor algorithm for computation of 
TEIs used in the solver II.     
We suppose that all basis functions $\{g_\mu \}_{1\leq \mu \leq N_b}$, 
are supported by the finite volume  box $\Omega=[-b,b]^3 \in \mathbb{R}^3$,
and assume, for ease of presentation, that $rank(g_\mu)=1$. 
Introducing the $n \times n \times n$ Cartesian grid over $\Omega$ and using the standard 
collocation discretization in the volume by piecewise constant 
basis functions, we define a grid-based tensor representation of the initial basis 
set $g_\mu(x) \in \mathbb{R}^3$, $\mu=1,\ldots N_b$,
\[
 g_\mu(x)=g_\mu^{(1)}(x_1)g_\mu^{(2)}(x_2)g_\mu^{(3)}(x_3) \approx {\bf G}_\mu 
=G_\mu^{(1)}\otimes G_\mu^{(2)} \otimes G_\mu^{(3)}\in \mathbb{R}^{n \times n \times n}.
\]
Define the respective product-basis tensor  
$$
\textsf{\textbf{G}}=[{\bf G}_{\mu\nu}] \in 
\mathbb{R}^{N_b\times N_b\times n^{\otimes 3}}
\quad \mbox{with} \quad {\bf G}_{\mu\nu}= {\bf G}_\mu \odot {\bf G}_\nu 
\in \mathbb{R}^{n^{\otimes 3}},
$$ 
where $\mu, \nu \in \{1,...,N_b\}$,
then both the TEI tensor and matrix  are represented by tensor operations,
\begin{equation} \label{gridb}
\textsf{\textbf{B}}= 
\textsf{\textbf{G}}\times_{n^{\otimes 3}}\textsf{\textbf{P}} \ast_{n^{\otimes 3}} \textsf{\textbf{G}},
\quad 
 b_{\mu\nu,\kappa\lambda}=
\left\langle {\bf G}_{\mu\nu} , \textsf{\textbf{P}}
\ast {\bf G}_{\kappa\lambda} \right\rangle_{n^{\otimes 3}}.
\end{equation}
Here the rank-$R_{\cal N}$ canonical tensor 
$\textsf{\textbf{P}}=\sum\limits_{k=1}^{R_{\cal N}}
{P}^{(1)}_{k}\otimes {P}^{(2)}_{k}\otimes {P}^{(3)}_{k} \in \mathbb{R}^{n^{\otimes 3}}$
approximates the Newton potential $\frac{1}{\|x\|}$ 
(see \cite{BeHaKh:08,KhKh3:08} for more details), 
$\ast$ stands for the 3D tensor convolution,
and $\odot$ denotes the Hadamard product of tensors.

Though tensor methods reduce the multidimensional integration to 1D complexity operations, 
the direct tensor-structured
evaluation of (\ref{gridb}) needs a storage size of at least, $O(R_{\cal N} N_b^2 n)$, which can be 
prohibitive for large $N_b\sim 10^2$  and  $n\approx 10^5$.
We apply the  RHOSVD-type factorization \cite{KhKh3:08} to the $4$th order
tensor $\textsf{\textbf{G}}$
by approximating its site matrices, $G^{(\ell)}\in \mathbb{R}^{n \times N_b^2}$, ($\ell=1,2,3$)
in a ``squeezed'' factorized  form, $G^{(\ell)} \cong U^{(\ell)} {V^{(\ell)}}^T$, 
according to the chosen $\varepsilon$-truncation. This step can be implemented 
by the truncated SVD in combination with incomplete truncated Cholesky decomposition.

This provides the construction of dominating subspaces in the $x$-, $y$- and $z$- components in 
the product basis set defined by an $n\times R_\ell$  matrix 
$U^{(\ell)}$ (left orthogonal basis) and $N_b^2 \times R_\ell$ matrix $V^{(\ell)}$ (right basis).
Then for the TEI matrix $B \in \mathbb{R}^{N_b^2 \times N_b^2}$, 
we obtain a factorization \cite{KhVBSchn_TEI:12,VeKhorMP2:13},
\begin{equation}\label{eq:BRF}
B\cong B_{\varepsilon}:= \sum\limits^{R_{\cal N}}_{k=1}\odot^3_{\ell=1} 
V^{(\ell)}  M_k^{(\ell)} {V^{(\ell)}}^T,
\end{equation}
where $V^{(\ell)}$ is the corresponding right redundancy-free basis, $\odot$ denotes the point-wise (Hadamard) 
product of matrices, and
\begin{equation}\label{eq:MRF}
M_k^{(\ell)}= {U^{(\ell)}}^T (P_k^{(\ell)}\ast_n U^{(\ell)})\in \mathbb{R}^{R_\ell \times R_\ell}, 
\quad k=1,...,R_{\cal N}.
\end{equation}
Ultimately, the TEI matrix $B$ is approximated in a form of the truncated Cholesky 
factorization, $B\approx L L^T$, $L\in \mathbb{R}^{N_b^2  \times R_B}$ ($R_B=O(N_b)$), 
such that the required columns of $B$ are easily computed by  using (\ref{eq:BRF}).

Vectorizing matrices $\overline{J}=vec(J(D))$, $\overline{K}=vec(K(D))$,
$\overline{D}=vec(D)$, we arrive at the simple
matrix representations,
\begin{equation}  \label{J_BD}
 \overline{J}= B \overline{D} \approx L (L^T \overline{D}),\quad 
vec(K) = \overline{K}= \widetilde{B} \overline{D},
\end{equation}
where $\widetilde{B}=mat(\widetilde{\textsf{\textbf{B}}})$ is the matrix unfolding of the permuted tensor 
$\widetilde{\textsf{\textbf{B}}}= [\widetilde{b}_{\mu \nu \kappa \lambda}]$ such that
$\widetilde{b}_{\mu \nu \kappa \lambda}= {b}_{\mu  \kappa \nu \lambda}$.

The nonlinear eigenvalue problem (\ref{SCF_EVP}) is solved by the commonly used 
DIIS self-consistent iteration
which requires the update of both Hartree and exchange operators at each iterative step. 

\subsubsection{Numerical illustrations}\label{sssec:HF_Numerics}

Figure \ref{fig:ala_energy}, left, illustrates the convergence history for self-consistent 
iterations by tensor solver II compared with the output of a standard quantum chemical 
package MOLPRO based on analytical calculations \cite{MOLPRO} for the Glycine amino acid,
$C_2 H_5 N O_2$.
The basis set cc-pVDZ of 170 Gaussians is used for both analytical and 3D grid-based calculations.
Here TEI is calculated on the grids $n^3=131072^3$. For this iteration the core Hamiltonian
is taken from MOLPRO. Time for one iteration is about $6$ seconds in Matlab.
\begin{figure}[htbp]
\centering
\includegraphics[width=5.2cm]{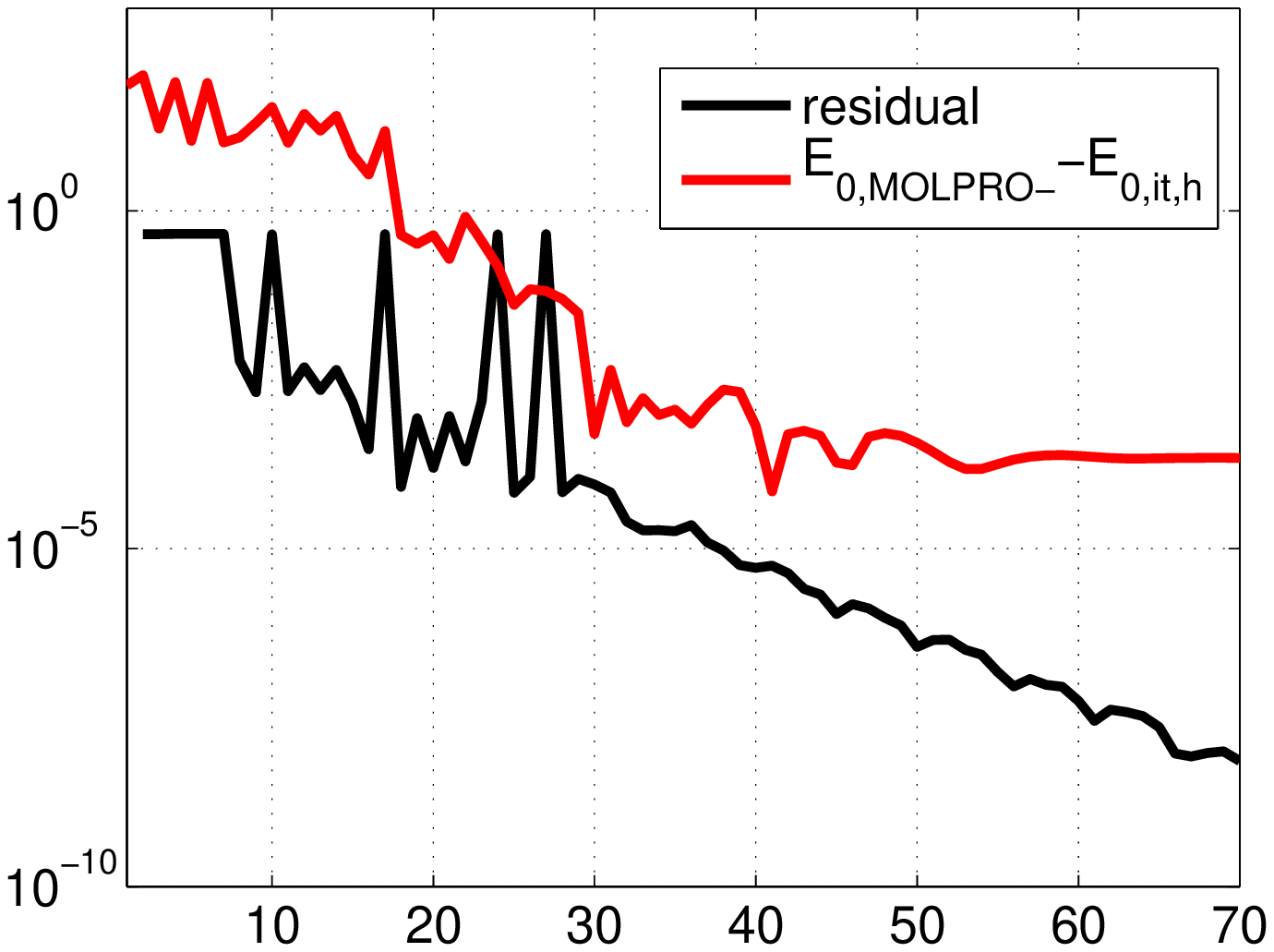} 
\includegraphics[width=5.0cm]{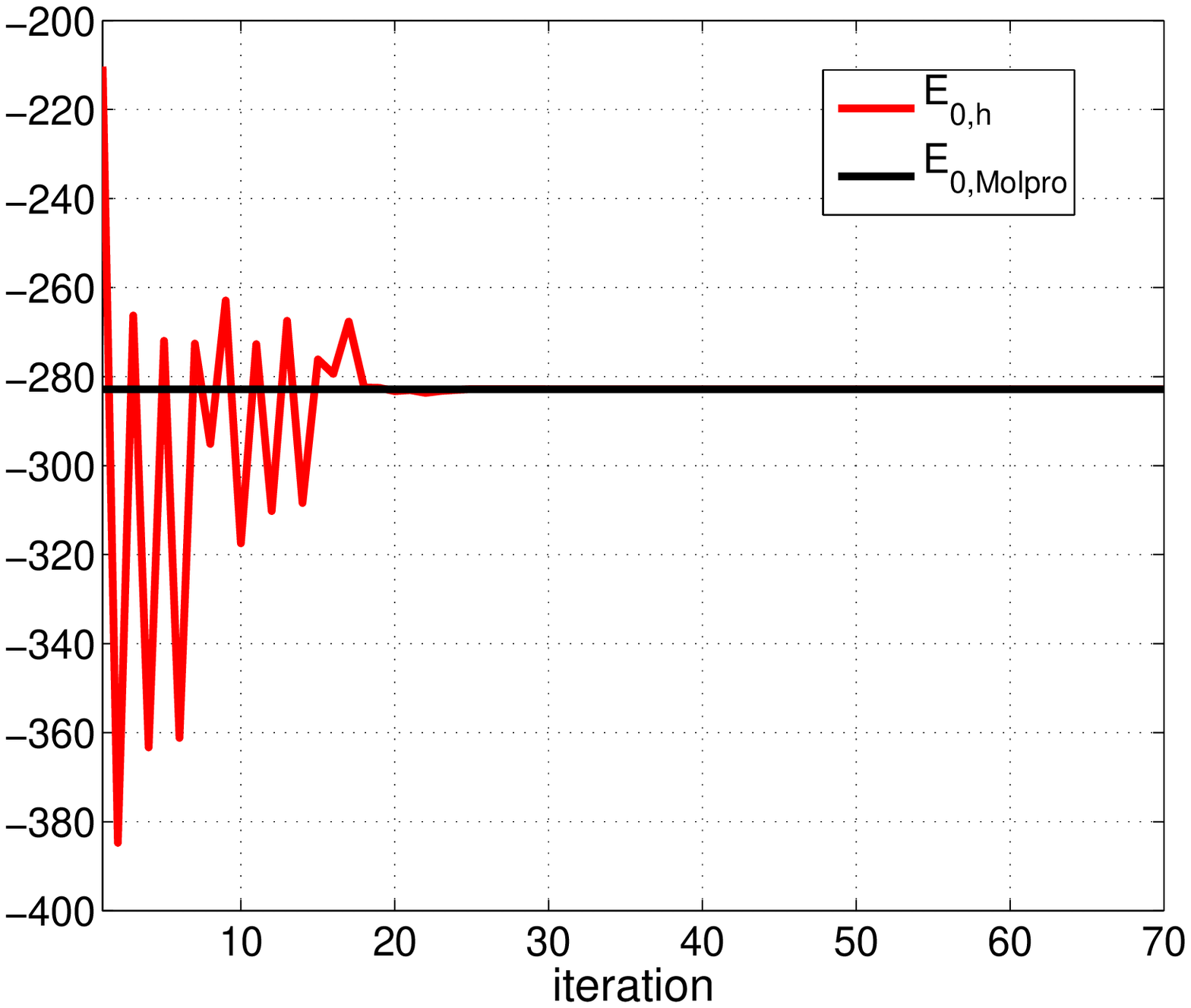}
\includegraphics[width=5.0cm]{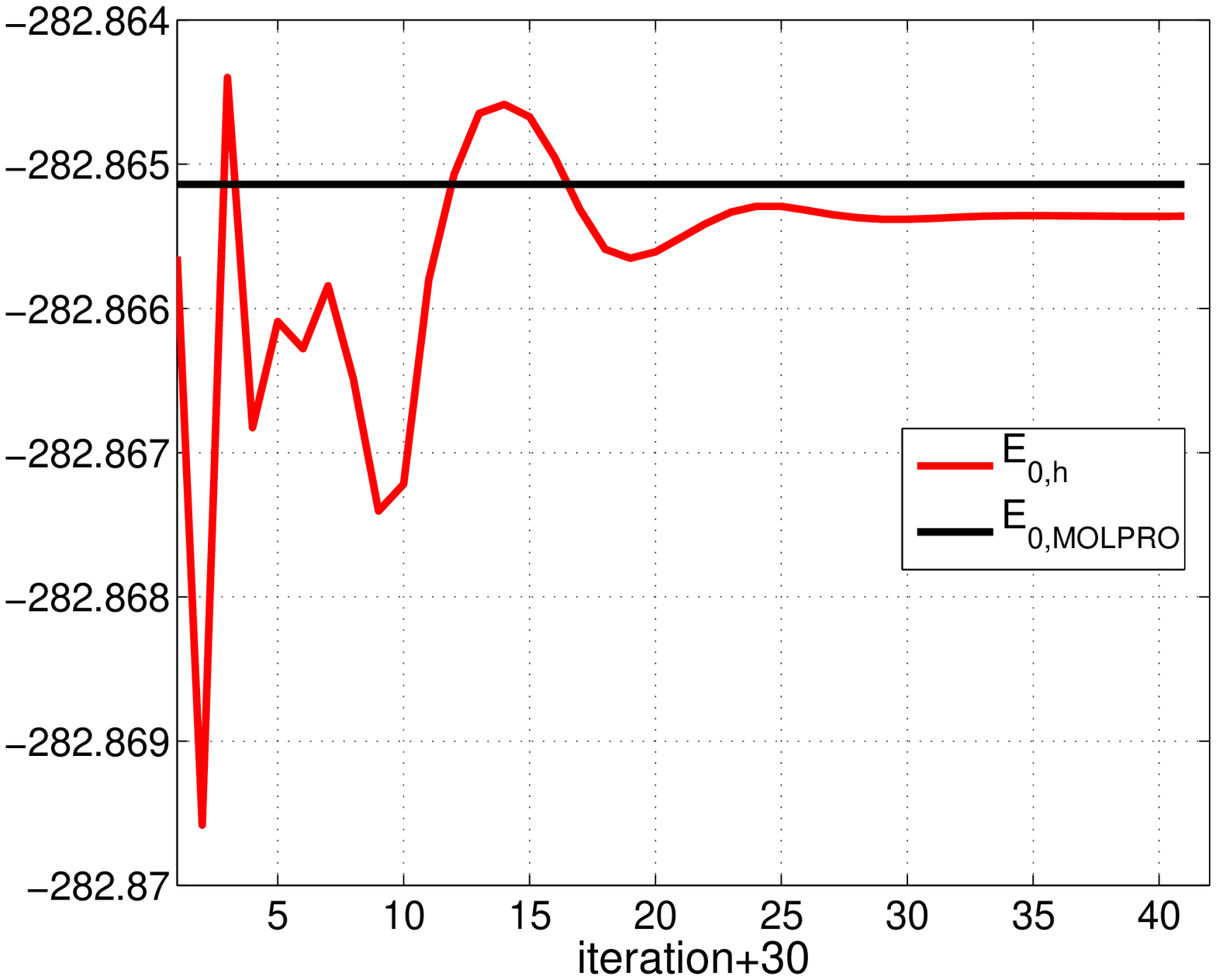}
 \caption{Left: iterations history for Glycine molecule with TEI calculated 
on the grids $n^3=131072^3$. Middle: convergence in energy. Right: a zoom 
for energy difference at last iteration.}
\label{fig:ala_energy}  
\end{figure}

We observe, that though the residual displays good convergence (in max-norm), 
the error with respect to
analytical calculations stagnates at $2\cdot 10^{-4}$ (relative error $< 7\cdot 10^{-7}$).
Figure \ref{fig:ala_energy} middle, shows the convergence in ground state energy $E_{0,it,h}$
while the right figure displays the zoomed difference of ground state energy for 
last iterations. Here the energy $E_{0,it,h}$ is computed by (\ref{eqn:ener_E0}) at each 
iteration step $it$ representing the convergence history.
The stagnation in the energy on lower than MOLPRO level (relative error $7\cdot 10^{-7}$)
may indicate the actual accuracy in computation of 3D convolution integrals in that code, 
and, beside, some possible instabilities of the grid-based algorithms 
applied on huge spatial grids ($n^3=131072^3$).
This topic needs further analysis to be conducted elsewhere. 

\begin{figure}[htbp]
\centering
\includegraphics[width=5.4cm]{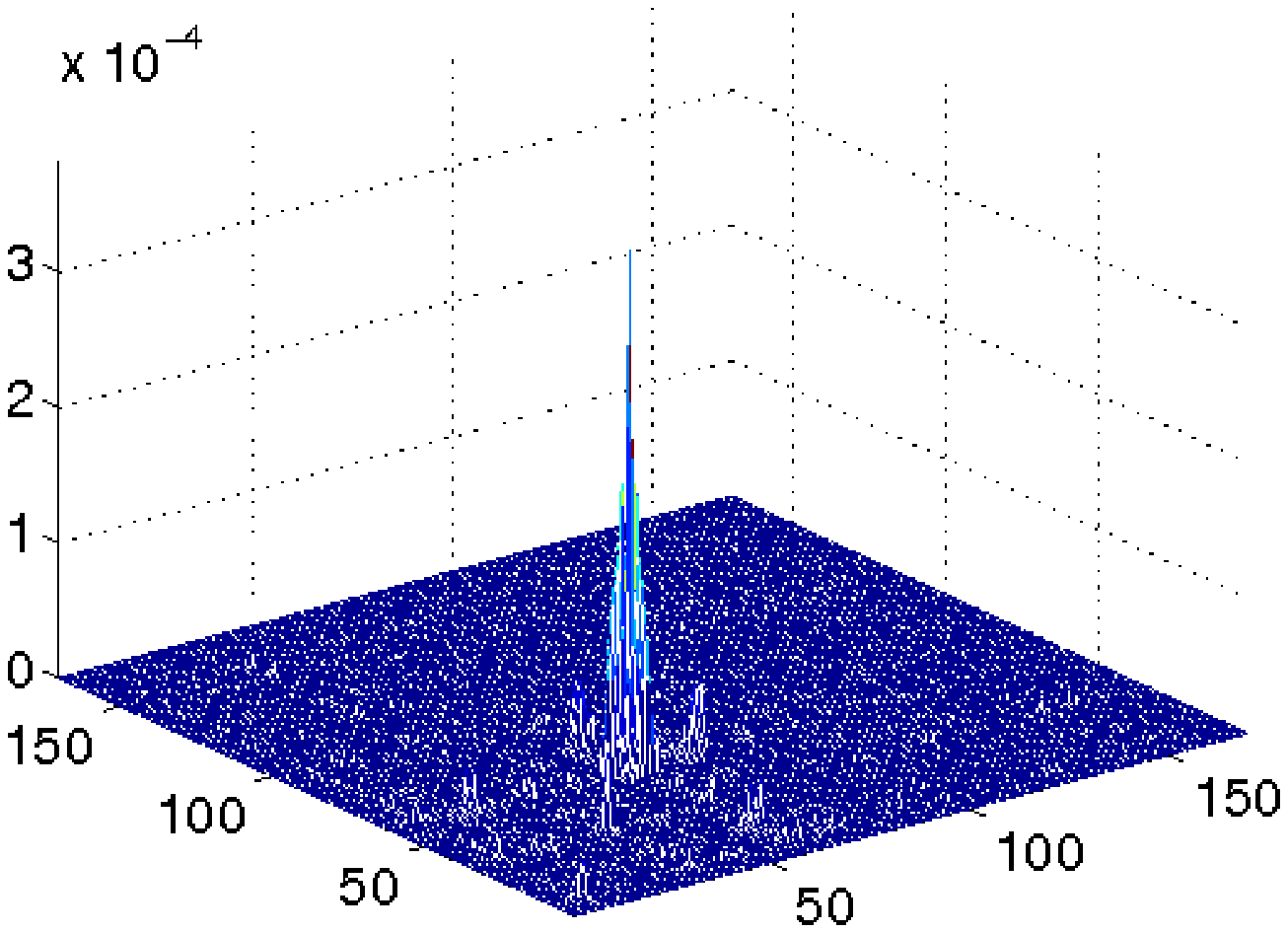} 
\includegraphics[width=5.4cm]{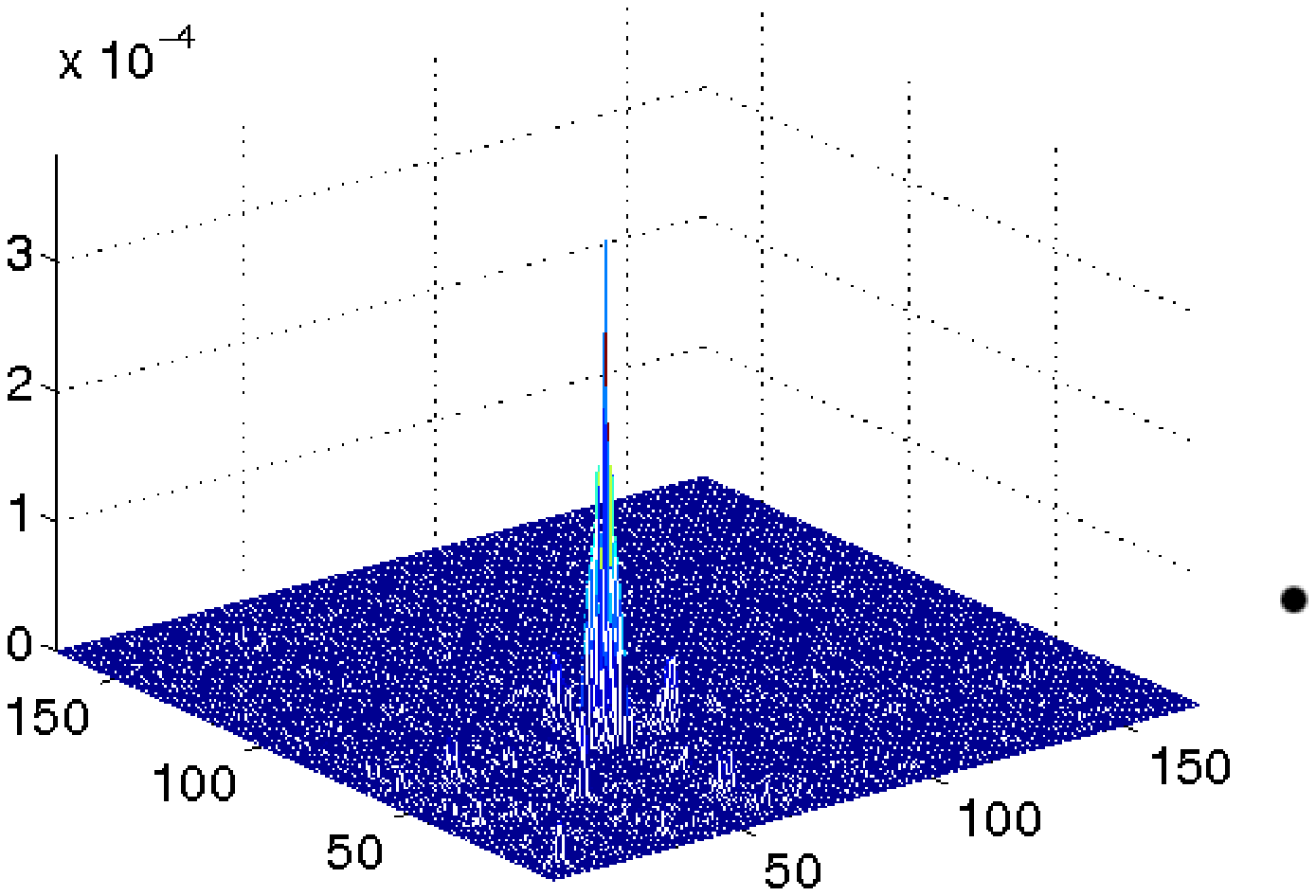}
\includegraphics[width=5.4cm]{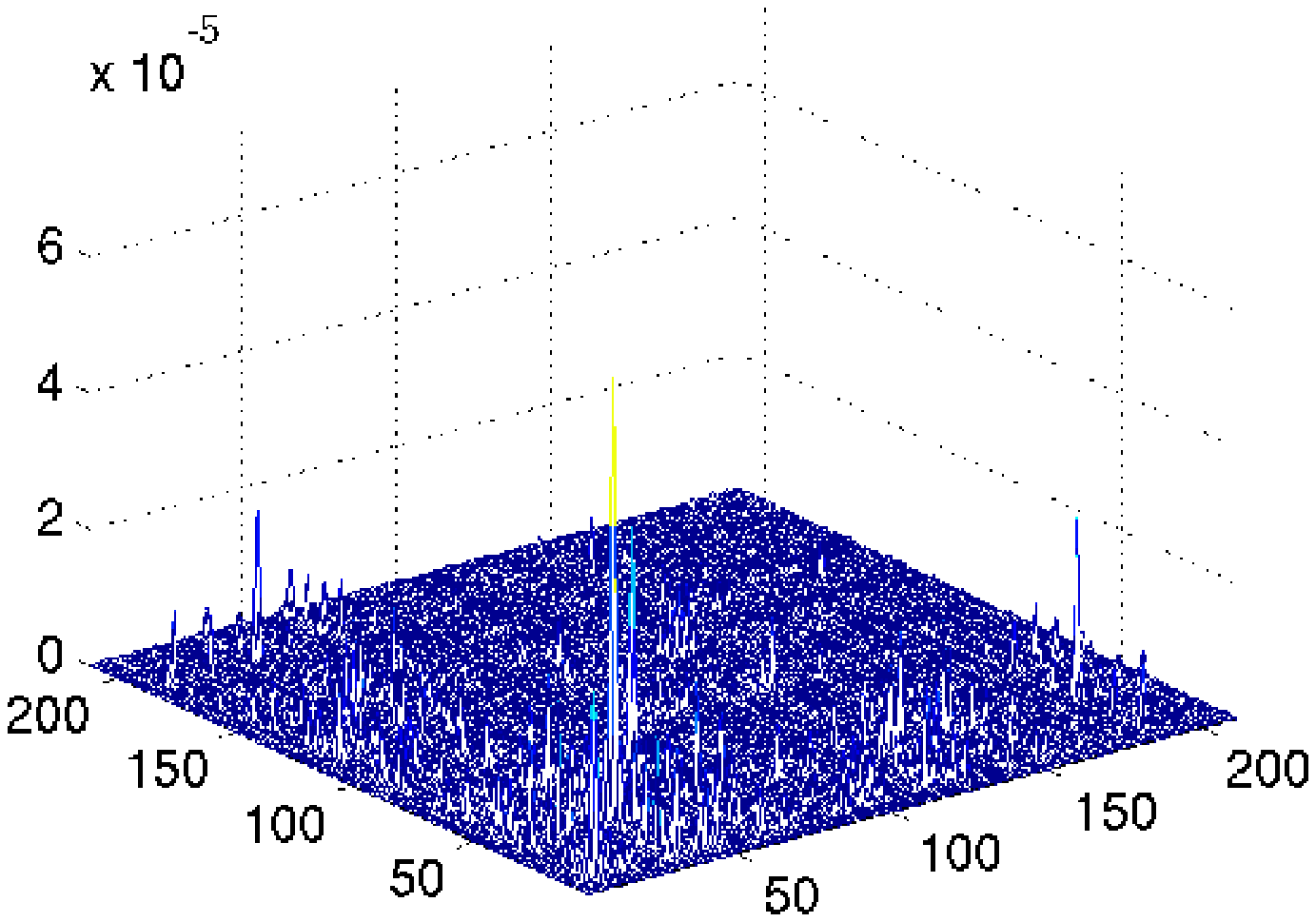}
\caption{Left: the error in density matrix resulting from calculations of TEI on
a grid of size $n^3=65536^3$. Middle: the same error for the grid size $n^3=131072^3$.
 Right: the error in density matrix for Alanine amino acid, TEI computed with $n^3=32768^3$.}
\label{fig:ala}  
\end{figure}

Figures \ref{fig:ala} (left, middle) show  the absolute error  (scaled by factor $10^{-4} - 10^{-5}$),
with respect to MOLPRO output,
of the tensor calculation of the density matrix for Glycine amino acid 
using TEI computed on the grids $n^3=65536^3$ (left) and $n^3=131072^3$ (middle).
Figure \ref{fig:ala} (right) shows the error of tensor calculations for Alanine amino acid,
$C_3 H_7 N O_2$,
using TEI computed on the grid $n^3=32768^3$, with $N_b^2=211^2$ basis functions.

The dominating part in the above tensor calculus resulted by rather large mode size $n$,
can be reduced to the logarithmic scale in $\log n$ by applying the QTT approximation.
Numerical illustrations on the QTT approximation of functions and operators 
arising in the solution of Hartree-Fock equation are given in Table \ref{2-fold-CH4}.
\begin{table}[h]
\begin{center}
\begin{tabular}{|c|c|c|c|c|}
\hline
$n^3$  & $2^{14\cdot 3}$ &  $2^{15\cdot 3}$ & $2^{16\cdot 3}$ &$2^{17\cdot 3}$  \\
\hline
$1/\|x\|$ & 3.15& 3.13 & 3.13 & 3.11 \\
\hline
$g_\mu g_\nu$ & 3.77 & 3.78  &  3.76 & 3.76 \\ 
 \hline
\end{tabular}
\end{center}
\vspace{0.2cm}
\caption{Average QTT-ranks for canonical vectors of the tensor $\textsf{\textbf{P}}$, 
and for the product basis  $\{g_\mu g_\nu\}$ designed for CH$_4$ molecule.}
\label{2-fold-CH4}
\end{table}
It indicates the low QTT-rank approximability of (a) the canonical vectors in low-rank 
decomposition $\textsf{\textbf{P}}$ of the Newton kernel $1/\|x\|$, in $\mathbb{R}^3$, 
see \cite{KhKh3:08},
(b) the product basis set designed for CH$_4$ molecule, both  discretized over large
$n \times n \times n$ spatial grids.
In all cases the  QTT approximation accuracy $\varepsilon=10^{-6}$ is achieved.
 We observe that QTT-ranks of canonical vectors for both the Newton  potential 
 and the product basis remain practically constant in $n$, ensuring 
 $O(\log \varepsilon^{-1}\log n)$ complexity scaling.

These results show that the grid-based tensor-structured Hartree-Fock solver II provides 
the accuracy and computation time compatible with the analytical calculations from MOLPRO.
More numerical results, including the complete grid-based calculations
with the grid-based core Hamiltonian and their comparison with MOLPRO are given 
in \cite{VKH_solver:13}. The Laplacian in core Hamiltonian is calculated on large 3D grids 
by using the QTT format.

We summarize  that tensor numerical methods described above are implemented in the 
Matlab program package Tensor-based Electronic Structure Calculation (TESC) by 
V. Khoromskaia and B. Khoromskij. TESC package allows the efficient grid-based solution of the
3D nonlinear Hartree-Fock equation discretized in a general set of basis functions 
characterized by the existence of low-rank separable representation.
All 3D and 6D integrals involved are approximated on large 
$n\times n \times n$ grids and computed by the 
black-box algorithms in the 1D complexity, $O(n)$, or even in $O(\log n)$ operations, 
which allows us the high resolution with the grid size up to $n=10^6$.

Further work should be focused, in particular, on the developments of the general 
type basis functions.
Now preliminary results are obtained for products of Gaussians with the plane waves.

\subsubsection{Tensor method for fast lattice summation}\label{sssec:EwaldSum}

The recent progress in tensor numerical methods for Hartree-Fock calculations is concerned with
the generalization of the above mentioned approach to the case of large lattice structured and periodic systems
\cite{VeBoKh:Ewald:14,VeKhorCorePeriod:14} arising in the modeling of cristalline, metallic and 
polymer type compounds. 

To fix the idea, we consider the electrostatic potential $V_c(x)$ in
(\ref{eqn:FockOper}) in the simplest case $M=1$.  Defined   
the scaled unit cell $\Omega_0=[-b/2,b/2]^3$,   of size $b\times b \times b$.
We consider a sum of interaction potentials in a symmetric box   
$$
\Omega_L =B\times B \times B,\quad \mbox{with} \quad B= \frac{b}{2}[- L  ,L ], \quad L\in \mathbb{N},
$$ 
consisting of a union of $L \times L \times L$ unit cells $\Omega_{\bf k}$,
obtained from $\Omega_0$ by a shift specified by the lattice vector $b {\bf k}$, where
${\bf k}=(k_1,k_2,k_3)\in \mathbb{Z}^3$, $-(L-1)/2  \leq k_\ell\leq (L-1)/2 $, 
($\ell=1,2,3$).
Here $L=1$ corresponds to a system in the unit cell. 
Recall that  $b=n h$, where $h>0$ is the mesh size that is the same for all spatial variables,
and $n$ is the number of grid points for each variable.
We also define the accompanying domain $\widetilde{\Omega}_L$
obtained by scaling of ${\Omega}_L$ with the factor of $2$,  
 $\widetilde{\Omega}_L= 2 {\Omega}_L$, and introduce the respective rank-$R$ master tensor
$\widetilde{\bf P}=\sum\limits_{q=1}^{R} 
\widetilde{P}^{(1)}_{q}\otimes \widetilde{P}^{(2)}_{q}\otimes \widetilde{P}^{(3)}_{q} $,
approximating $\frac{1}{\|{x} \|}$ in $\widetilde{\Omega}_L$ on tensor grid with mesh size
$h$.

In the case of extended system in a box the summation 
problem for the total potential $V_{c_L}$ is formulated in the domain 
$\Omega_L= \bigcup_{k_1,k_2,k_3=-(L-1)/2}^{(L-1)/2} \Omega_{\bf k}$ as well as in the 
accompanying domain. 
On each $\Omega_{\bf k}\subset \Omega_L$, the  target potential  
$v_{\bf k}(x)=(V_{c_L})_{|\Omega_{\bf k}}$, 
is obtained by summation over all unit cells in $\Omega_L$,
\begin{equation*}\label{eqn:EwaldSumE}
v_{\bf k}(x)=   \sum\limits_{k_1,k_2,k_3=-(L-1)/2}^{(L-1)/2} 
\frac{Z_0}{\|{x} - b {\bf k} \|}, \quad x\in \Omega_{\bf k}. 
\end{equation*}

This calculation is performed at each of $L^3$ elementary cells 
$\Omega_{\bf k}\subset \Omega_L$ simultaneously, which is implemented by the assembled 
tensor summation method described in \cite{VeBoKh:Ewald:14}. The resultant lattice sum is
presented by the canonical rank-$R$ tensor 
${\bf P}_{c_L} \in \mathbb{R}^{nL\times nL\times nL}$, 
\begin{equation}\label{eqn:EwaldTensorGl}
{\bf P}_{c_L}=  Z_0 
\sum\limits_{q=1}^{R}
(\sum\limits_{k_1=-(L-1)/2}^{(L-1)/2}{\cal W}_{{k_1}} \widetilde{P}^{(1)}_{q}) \otimes 
(\sum\limits_{k_2=-(L-1)/2}^{(L-1)/2}{\cal W}_{{k_2}} \widetilde{P}^{(2)}_{q}) \otimes 
(\sum\limits_{k_3=-(L-1)/2}^{(L-1)/2}{\cal W}_{{k_3}} \widetilde{P}^{(3)}_{q}),
\end{equation}
where ${\cal W}_{{k_\ell}}$ is the shift-and-windowing transform along the ${\bf k}$-grid.
The numerical cost and storage size are bounded by $O(R L N_L )$, 
and $O(R N_L)$, respectively (see \cite{VeBoKh:Ewald:14}, Theorem 3.1), where  $N_L= nL$,
and $n$ is the grid size in the unit cell. 

The lattice sum in (\ref{eqn:EwaldTensorGl}) converges only conditionally as $L\to \infty$.
This aspect was addressed in \cite{VeBoKh:Ewald:14}.

\begin{figure}[htbp]
\centering
\includegraphics[width=5.0cm]{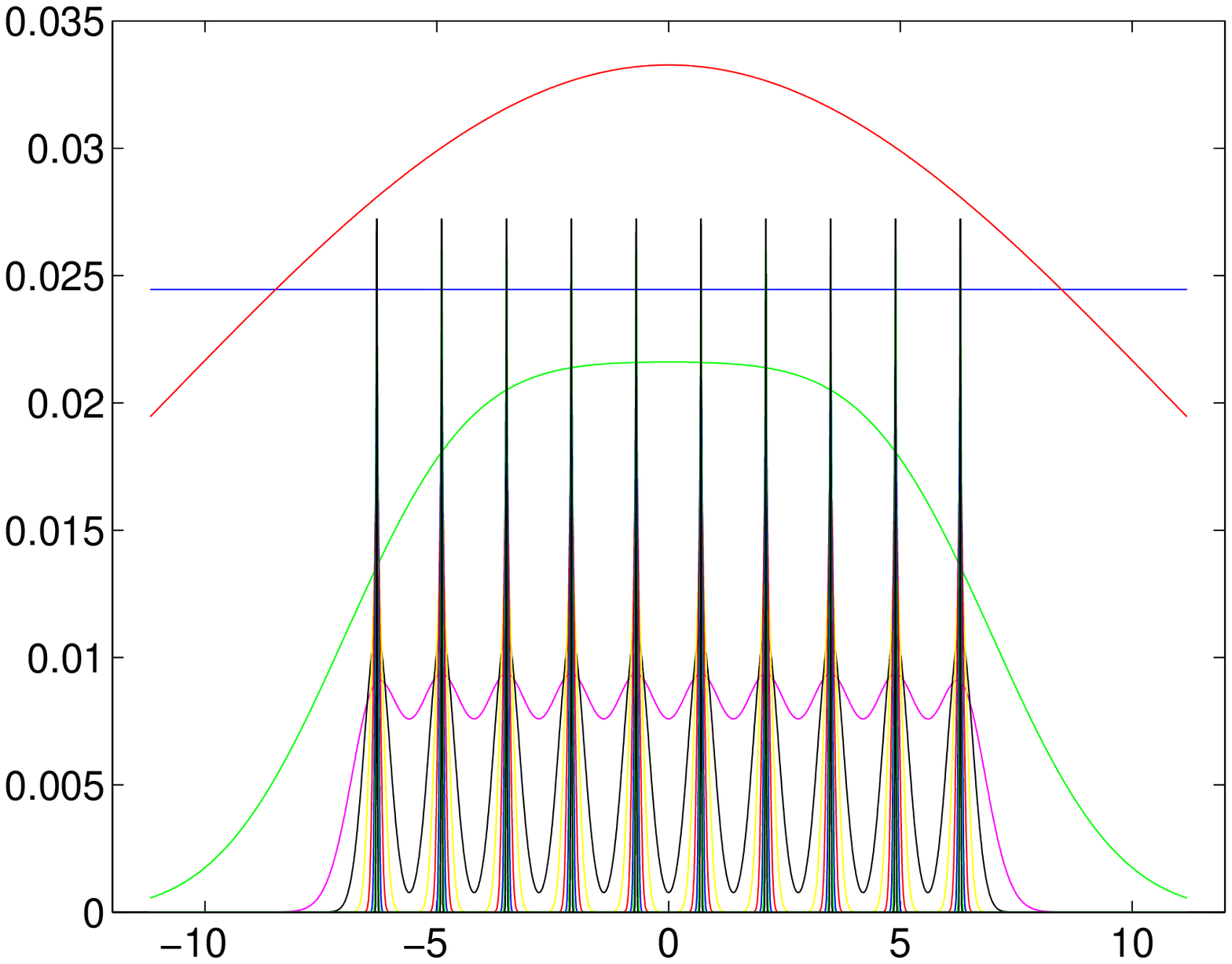}\quad
\includegraphics[width=5.0cm]{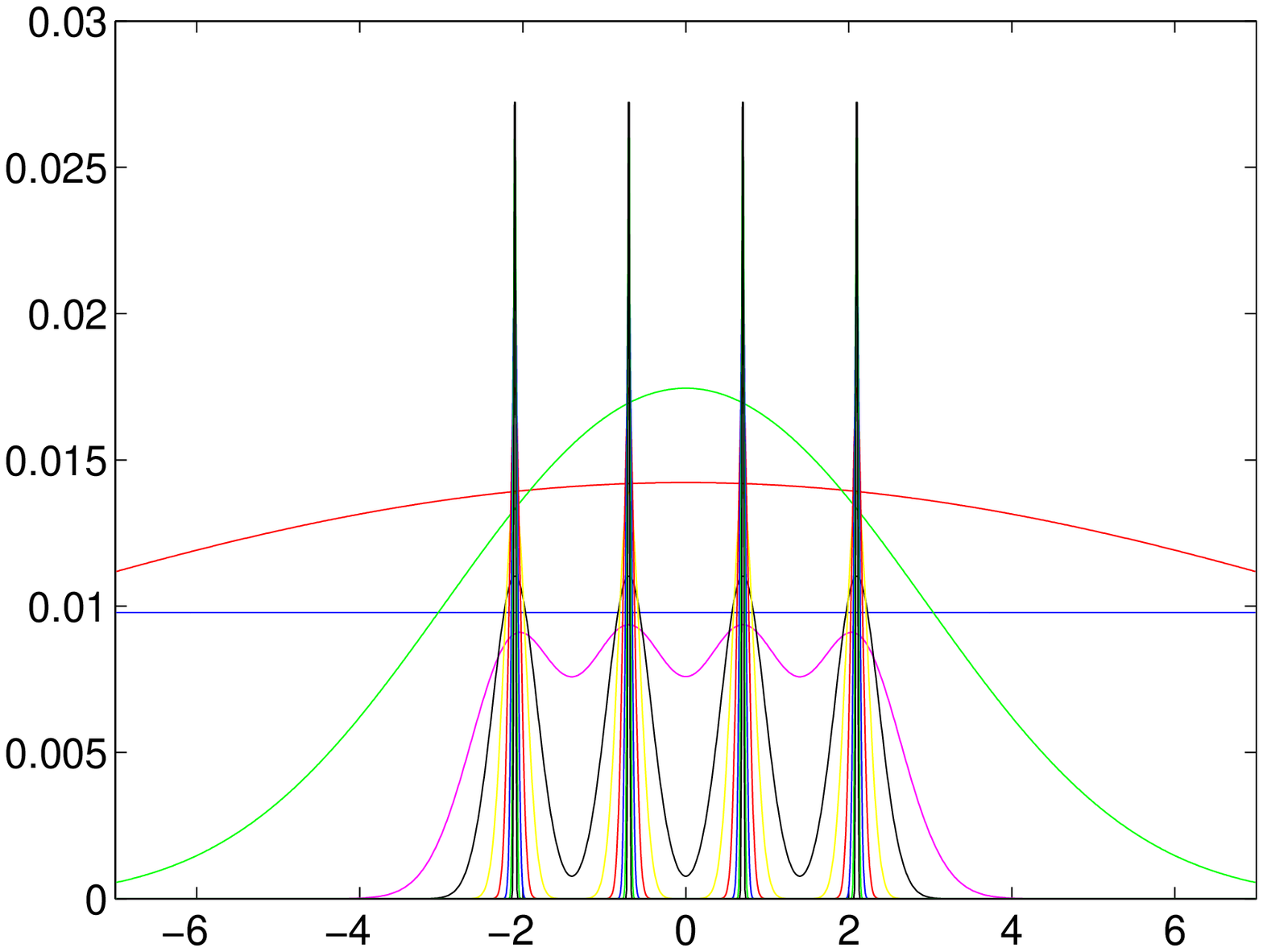}\quad
\includegraphics[width=5.0cm]{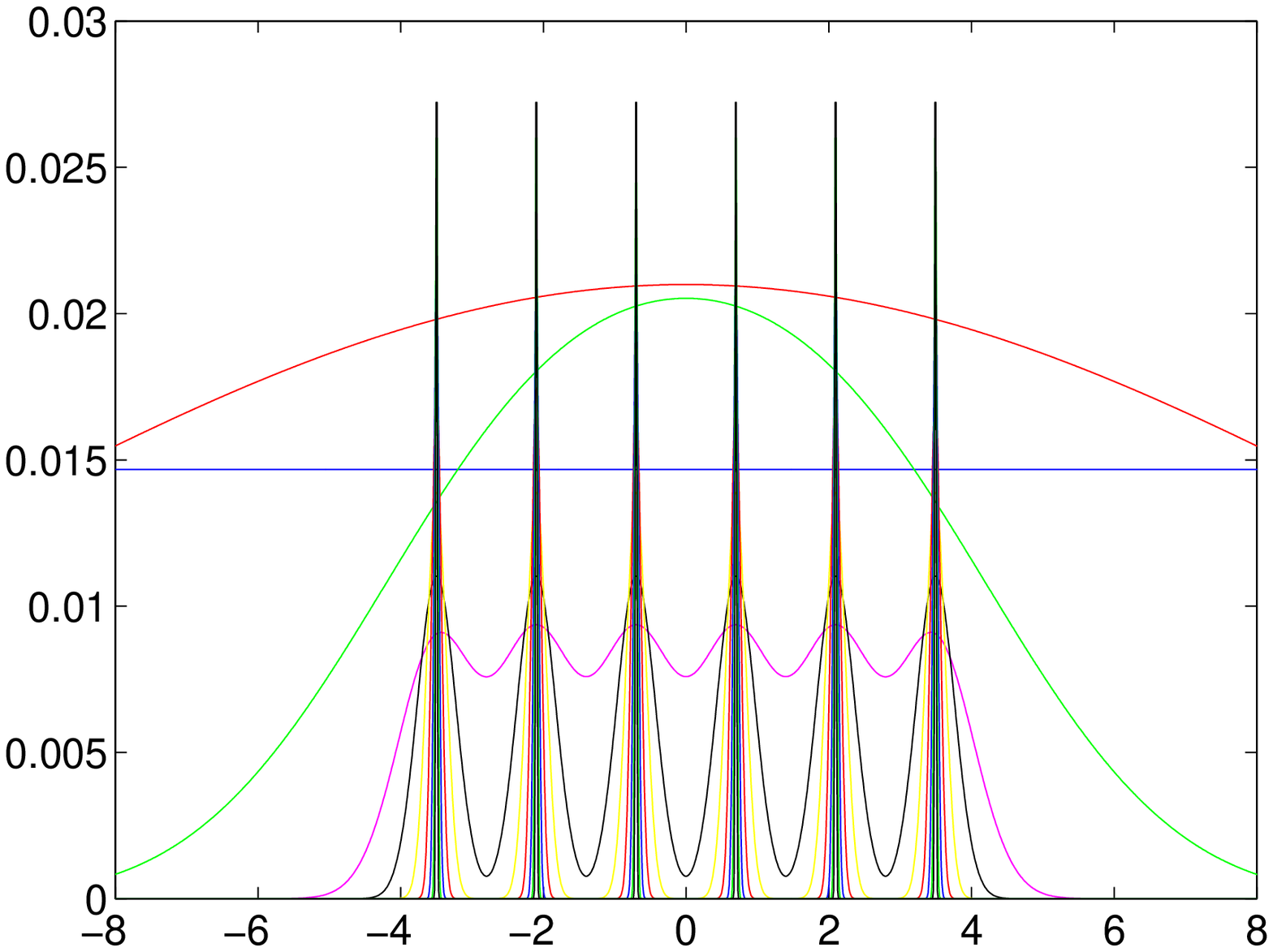}
 \caption{Canonical vectors in the assembled $10\times 4 \times 6$ lattice sum.}
\label{fig:ewald_sum}  
\end{figure}

Figure \ref{fig:ewald_sum} demonstrates the canonical vectors for assembled tensor sum
corresponding to $10\times 4 \times 6$ lattice.

\subsection{Real-time dynamics by parabolic equations}\label{ssec:dynamics}

\subsubsection{General introduction}\label{sssec:IntrodDynamics}

Let $\mathbb{W}$ be a complex Hilbert space and ${\cal H}$ be a self-adjoint
positive definite operator with the domain $D({\cal H})$ and the spectrum
$\Sigma({\cal H}) \in [\lambda_0, \infty), \lambda_0 >0$. 
Given $\sigma\in \{-1,i\}$, we consider the
following initial value problem 
\begin{equation}\label{eqn:1}
  \frac{\partial \psi}{\partial t}=\sigma{\cal H} \psi(t), \quad \psi(0)=\psi_0\in
  D({\cal H})\subset \mathbb{W}.
\end{equation}
The solution of (\ref{eqn:1}) is represented by using operator exponential,
$
\psi(t)= e^{\sigma {\cal H}t} \psi_0,
$
however, in general, the solution operator of this parabolic problem, $S(t)=e^{\sigma{\cal H}t}$, 
does not allow the accurate low-rank tensor approximation.  

In  quantum mechanics, equation like (\ref{eqn:1}) may 
represent the molecular or electronic Schr{\"o}dinger equation in $d$
dimensions that describes how the quantum state of a physical 
system evolves in time. 
In this case the many particle Hamiltonian ${\cal H}$ is given by a
sum of $d-$dimensional Laplacian and certain interaction potential, 
\cite{LubichBook:08}, 
   \begin{equation*}\label{eq_mol1}
      i \frac{\partial \psi}{\partial t} = {\cal H} \psi =
      (-\frac{1}{2}\Delta_d + V) \psi , \quad \psi(x,0) = \psi_0(x),\quad x\in
      \mathbb{R}^d,
    \end{equation*}
where $V:\mathbb{R}^d \to \mathbb{R}$ is (given) approximation to the potential energy surface (PES).

The interesting example of the real-time dynamics 
is given by the Fokker-Planck equation, which is usually high-dimensional.
It models the joint probability density distribution of noisy dynamical system configurations 
(e.g. positions of particles). 
The initial (stochastic) system of ODEs reads
\begin{equation*}
\dfrac{dx}{dt} = -A(x) + G\eta \in \mathbb{R}^d,
\label{rake:eqn:ode}
\end{equation*}
where the noise satisfies $<\eta>=0$, and $<\eta_i \eta_j> = \delta_{ij}$.
The probability to find configurations in some volume $x^*+dx$ is written as follows,
$$
P\left(x \in \mathbb{B}_{|dx|}(x^*)\right) = \psi(x^*) dx,
$$
and the \emph{deterministic} real-time parabolic PDE on the probability 
density, called the Fokker-Planck equation, reads as
\begin{equation*}
\psi(0)=\psi_0: \quad
\dfrac{\partial\psi}{\partial t} = \dfrac{\partial}{\partial x} \cdot \left(A(x) \psi \right) + 
\dfrac{1}{2}\dfrac{\partial}{\partial x} \cdot \left(D \dfrac{\partial\psi}{\partial x} \right),
\quad \mbox{where}\quad D = GG^T.
\label{rake:eqn:FPE}
\end{equation*}
Here $\psi:\mathbb{R}^d \rightarrow \mathbb{R} $, and 
${\bf v}: \mathbb{R}^d \rightarrow \mathbb{R}^d$ 
is a given velocity field. In many cases the computation of stationary solution 
$\psi(t)\to \psi_\ast: A \psi_\ast =0$, is the main target. 
We refer to \cite{DoKhOsel:11} for detailed discussion and numerical tests.

In Section  \ref{sssec:CME_QTTTuck} we discuss in more detail the tensor numerical scheme
for the solution of chemical master equation.
This model describes the dynamics of joint probability density ${\cal P}({\bf x},t)$, 
\begin{equation}\label{CME:eq1}
{\cal P}({\bf x},0) = {\cal P}_0,\quad 
\frac{d {\cal P}({\bf x},t)}{d t} = {\bf A} {\cal P}({\bf x},t),  
\quad  {\bf x}\in \mathbb{R}^{n_1\times ... \times n_d},
\end{equation}
where ${\cal P}({\bf x},t)$ is the joint probability of the numbers of molecules of species 
$S_1,...,S_d$, reacting in $M$ channels, to take particular values $x_1,...,x_d$ at time $t$.

\subsubsection{Rank bounds for Cayley transform in space-time tensor approximation}\label{sssec:CayleyDyn}

In \cite{GavlKh_Caley:11} the QTT-Cayley transform was proposed to compute
dynamics and spectrum of high-dimensional Hamiltonians with the focus 
on complex-time dynamics. Here, using the similar techniques, we analyze in more
detail the case of real-time dynamics, i.e. $\sigma=-1$ in (\ref{eqn:1}),
already sketched in \cite{GavlKh_Caley:11}. 
For the ease of presentation, we further assume that the self-adjoint Hamiltonian operator ${\cal H}$ 
has the complete eigenbasis, $\mathbb{W}=\operatorname{span}\{\phi_n\}_{n=0}^\infty$, with 
the real eigenvalues $ 0< \lambda_0\leq \lambda_1\leq ... $. 
An extension to the more general class of convection-diffusion operators as in (\ref{rake:eqn:FPE}) 
is possible and it will be addressed elsewhere.

The idea on separation of the time and space variables via Cayley transform 
is based on the series expansion for the solution operator  
\begin{equation}\label{eqn:2}
  e^{-{\cal H}t}=({\cal H}+ I)^{-1}
  \underset{p=0}{\overset{\infty}{\sum}}L_p (t){\cal C}^{p}({\cal H}),
\end{equation}
where
$
{\cal C}={\cal C}({\cal H})={\cal H}({\cal H}+ I)^{-1}
$ 
is the  Cayley transform of the operator ${\cal H}$, 
and $L_p(t) = L_p^{(0)}(t)$ is the Laguerre polynomial of degree $p$ \cite{szegoe,bateman}.
Expansion (\ref{eqn:2}) applies (convergence in $\mathbb{W}$ norm) 
to every initial vector $\psi_0\in D({\cal H})$, i.e.,
\begin{equation}\label{eq init x0}
\psi_0=\sum\limits_{k=0}^\infty a_k \phi_k, \quad \mbox{such that}\quad  
\sum\limits_{k=0}^\infty |a_k|^2 \lambda_k^2 < \infty.
\end{equation}


Separation of the time variable $t$ from the spatial part of the solution
is based on the observation that the solution of our initial value problem subject
(\ref{eq init x0}) can be represented as
\begin{equation}\label{3}
\psi(t)=\underset{p=0}{\overset{\infty}{\sum}}L_p (t)u_{p} 
\equiv ({\cal H}+ I)^{-1}
\underset{p=0}{\overset{\infty}{\sum}}L_p (t){\cal C}^{p} \psi_0,
\end{equation}
where the elements $u_p$ can be found from the recursion
\begin{equation*}\label{5}
u_0=({\cal H}+ I)^{-1}\psi_0: \quad u_{p+1}={\cal H}({\cal H}+ I)^{-1}u_p, \quad p=0,1,...
\end{equation*}
Now, as a computable approximation to the exact solution, we
consider the $m$-term truncated series representation
\begin{equation}\label{eqn:15}
\begin{split}
  \psi_m(t)&= ({\cal H}+ I)^{-1}
\underset{p=0}{\overset{m}{\sum}}L_p (t) {\cal C}^{p} \psi_0=
\psi_0 + {\cal H}\underset{p=0}{\overset{m}{\sum}}
  (L_{p+1} (t)-L_{p} (t))u_{p},
\end{split}
\end{equation}
that effectively separates space and time variables.

Let us show that approximation (\ref{eqn:15}) leads to an exponential convergence
rate in $m$ for the ${\cal H}$-analytical input data.
\begin{definition}\label{def B-anal}
A vector $f= \sum\limits_{k=0}^\infty a_k \phi_k \in D({\cal H})$ is called analytical for 
${\cal H}$ (${\cal H}$-analytic)
if there is a constant $C=C(f) >0$, such that 
  $$
 \|{\cal H}^n f\|= \sqrt{\sum\limits_{k=0}^\infty |a_k|^2 \lambda_k^{2n}}
  \leq C^n n! \quad \mbox{for all} \quad n=1,2,3,...
  $$
\end{definition}
Next theorem proves the exponential convergence of the $m$-term approximation (\ref{eqn:15}),
see  \cite{GavlKh_Caley:11}, Remark 2.9.
\begin{theorem}\label{thm:exp_conv_TGT}
Let $\psi_0$ be ${\cal H}$-analytic and let $r>0$ be 
the convergence radius of the series 
$\underset{k=0}{\overset{\infty}{\sum}}\frac{s^k}{k!}\|{\cal H}^k \psi_0\|$.
Then for every fixed $0< s<r$, and  $T> 0$, there exist $c,c_1 >0$ independent of $m$, 
such that for all $m\in \mathbb{N}$,
\begin{equation}\label{18}
\begin{split}
  &\|\psi(t)-\psi_m(t)\|\le 
c t^{\frac{1}{4}}e^{\frac{t}{2}}m^{-1/4}
  e^{-c_1\sqrt{m}}\|\psi_0\|_{s,{\cal H}},
  \quad t\in[0,T],
\end{split}
\end{equation}
where $\|\psi_0\|_{s,{\cal H}}:=\underset{k=0}{\overset{\infty}{\sum}}\frac{s^k}{k!}\|{\cal H}^k \psi_0\|$.
\end{theorem}
\begin{proof}
  First, we note that the asymptotic properties of Laguerre polynomials 
yield 
\begin{equation}\label{19-1}
\begin{split}
  &\|\psi(t)-\psi_m (t)\|\le c t^{-\frac{1}{4}}e^{\frac{t}{2}}
  \underset{p=m+1}{\overset{\infty}{\sum}}p^{-3/4}\|u_p\|,\quad \mbox{for}\quad t \in [\varepsilon, T],
\end{split}
\end{equation}
where the iterand 
$u_{p+1}= \sum\limits_{k=0}^\infty a_k \bigg(\frac{\lambda_k}{\lambda_k+1}\bigg)^{p} \phi_k$ 
admits the representation
\begin{equation}\label{eq B-anal err}
 \begin{split} 
  u_{p+1}
  &= \sum\limits_{k=0}^\infty a_k e^{-\lambda_k s}
 \bigg(\frac{\lambda_k}{\lambda_k+1}\bigg)^{p} 
\bigg( \sum\limits_{n=0}^\infty  \frac{\lambda_k^n s^n}{n!} \bigg) \phi_k,\\
&= \sum\limits_{k=0}^\infty a_k e^{-\lambda_k s}
 \bigg(\frac{\lambda_k}{\lambda_k+1}\bigg)^{p} 
\sum\limits_{n=0}^\infty \frac{s^n}{n!} {\cal H}^n \phi_k,\\
&= \sum\limits_{n=0}^\infty \frac{s^n}{n!} {\cal H}^n 
\bigg( \sum\limits_{k=0}^\infty a_k \Phi_s(\lambda_k) \phi_k\bigg),
\end{split}
\end{equation} 
with
$
\Phi_s(\lambda):= e^{-\lambda s} \left(\frac{\lambda}{\lambda+1}\right)^{p}. 
$
The simple variational analysis indicates that the function  
$\Phi_{s}(\lambda)$ takes its maximum at a point $\lambda_* \asymp \sqrt{p}$.
Hence, taking into account that 
\[
\|{\cal H}^n \bigg( \sum\limits_{k=0}^\infty a_k \Phi_s(\lambda_k) \phi_k\bigg) \|
\leq   \max\limits_{\lambda\in [\lambda_0,\infty)}
 \bigg| \Phi_s(\lambda) \bigg| \|{\cal H}^n  \psi_0\|,
\]
we arrive at the estimate 
$\|u_{p+1}\|\le c e^{-c_1 \sqrt{p}}\|\psi_0\|_{s,{\cal H}}$, implying 
\begin{equation*}\label{9-2-1}
\begin{split}
  \|\psi(t)-\psi_m(t)\|  & \le c
  t^{\frac{1}{4}}e^{\frac{t}{2}}\|\psi_0\|_{s,{\cal H}}
  \underset{p=m+1}{\overset{\infty}{\sum}}p^{-1/4}p^{-1/2}e^{-c_1\sqrt{p}} \\
  &\le c t^{\frac{1}{4}}e^{\frac{t}{2}}m^{-1/4}
  e^{-c_1\sqrt{m}}\|\psi_0\|_{s,{\cal H}}, 
\end{split}
\end{equation*}
which completes our proof.
\end{proof}
Notice that the constant $c_1 \approx s^{1/2}$ depends on $s$, while $c$ does not 
(see \cite{GavlKh_Caley:11} for the discussion in the complex case $\sigma=i$).

In the following discussion we consider the semi-discrete scheme 
(i.e., already discretized in space), such that the operators ${\cal H}$ and ${\cal C}$
are substituted by a matrix ${\bf H}$ and ${\bf C}$, respectively. 
Assume that $\psi_m(t)\in \mathbb{W}_{\bf n}$ 
represents a $d$th order tensor obtained by the truncated series 
representation (\ref{eqn:15}) composed of the discretized solutions $u_p$,
$p=0,1,...,m$, 
and let $t_0,...,t_{N_t}\in [0,T]$ be the uniform discretization grid in time
with a step size $\tau$.

Given the rank-truncation threshold $\varepsilon > 0$, then
similar to Lemma 3.4 in \cite{GavlKh_Caley:11}, we derive that
the choice $m=O(\log^2 \frac{1}{\varepsilon})$ implies that
the  $\varepsilon$ QTT-rank of a concatenated tensor 
$$
\textsf{\textbf{P}}_m=
  [\psi_m(t_0),...,\psi_m(t_{N_t})]_{k=0}^{N_t}\in
\mathbb{W}_{\bf n}\times \mathbb{R}^{N_t+1}, \quad t_k=k \tau,
$$
obtained by sampling of $\psi_m(t)$ on the time grid, is bounded by 
$$
rank_{QTT}(\textsf{\textbf{P}}_m) \leq \sum\limits_{p=0}^{m} 
(p+1)rank_{QTT} ({{\bf{C}}}^p \psi_0)
\leq C m^2 rank_{QTT}({{\bf{C}}}^m \psi_0). 
$$ 

Suppose that $rank_{QTT}({\bf{C}}^m \psi_0)$ is small, then
the block two-diagonal system of equations defined by the implicit Euler scheme,
\begin{equation}\label{x-t system}
\psi_0=\psi(0), \quad ({\bf{I}}+ \tau  {{\bf{H}}})\psi_{k+1}- \psi_k= 0, 
\quad k=0,1,...,N_t-1,
\end{equation}
where $\psi_k\in\mathbb{W}_{\bf n} $ will approximate the value 
of the true solution $\psi_m(t_k)$,
can be assumed to have a low QTT-rank solution represented by tensor $\textsf{\textbf{P}}$,
with $O(d \log N \log N_t)$ complexity scaling.
In turn, (\ref{x-t system}) can be solved in the QTT format as the global system of equations with
respect to the unknown space-time vector (tensor)
$$
\textsf{\textbf{P}}=[\psi_0,\psi_1,...,\psi_{N_t}]
\in \mathbb{W}_{\bf n}\times \mathbb{R}^{N_t+1}\approx
\textsf{\textbf{P}}_m.
$$
The solution of the global $(x,t)$ system (\ref{x-t system}) living in large
virtual dimension $D = d \log N \log N_t$ can be approached by
either tensor-truncated preconditioned or AMEn/DMRG type iteration 
with the asymptotic cost $O(d \log N \log N_t)$
(see \cite{DoKhOsel:11,DoKh_CME_NLAA:13,DoSav_AMEn:13} for more detail).

\subsubsection{Chemical master equation in the QTT-Tucker format}\label{sssec:CME_QTTTuck}

In this section, we discuss the tensor numerical scheme via the QTT-Tucker format
\cite{DoKh_CME_NLAA:13} for the solution of chemical master equation (CME).

Suppose that $d$ species $S_1,...,S_d$ react in $M$ reaction channels.
Denote the vector of their concentration $\mathbf{x}=(x_1,...,x_d)$, $x_i\in \{0, ..., N_i-1\}$.
Each channel is specified by a stoichiometric vector 
$\mathbf{z}^{m}\in \mathbb{Z}^d$, where $\mathbf{z}=(z_1,...,z_d)$,
and a propensity function $w^m(\mathbf{x}), m=1,...,M$.
Introduce the shift matrices
\begin{equation*}
J^{z} = \begin{bmatrix}
         0 & \cdots & 1 \\
           & \ddots & & \ddots \\
           & & \ddots & & 1 \\
           & & & \ddots & \vdots \\
         & &  &  & 0
        \end{bmatrix} \begin{matrix} \phantom{0} \\ \phantom{\ddots} \\ \leftarrow & \mbox{row} & N-z \\ \phantom{\ddots} \\ \leftarrow & \mbox{row} & N\phantom{-z,} \end{matrix},
        \quad \mbox{if}~z \ge 0;\qquad  J^{z} = (J^{-z})^\top, \mbox{if }  z<0.
\label{cme:eqn:shift_matrices}
\end{equation*}
Now  the finite state approximation (FSP) of \eqref{CME:eq1} can be written as a linear ODE,
the so-called CME, that  is a deterministic difference equation on the joint 
probability density $P(x,t)$:
\begin{equation*}
\dfrac{dP(t)}{dt} = A P(t), \quad 
A= \sum\limits_{m=1}^M (\mathbf{J}^{\mathbf{z^m}} - \mathbf{J}^0) \diag(w^m) P(t), 
\qquad P(t) \in \mathbb{R}_+^{{\prod_{i=1}^d N_i}},
\label{cme:eqn:cme-shift}
\end{equation*}
\begin{equation*}
 \mathbf{J}^{\mathbf{z}} = J^{z_1} \otimes \cdots \otimes J^{z_d},
\end{equation*}
$w^m = \{w^m(\mathbf{x})\}$ and $P(t) = \{P(\mathbf{x},t)\}$, 
$\mathbf{x} \in \bigotimes\limits_{i=1}^d \{0,\ldots,N_i-1\}$,
are the corresponding values of $w^m$ and $P$ stacked into vectors,
$\diag(w^m)$ is a diagonal matrix with the values of $w^m$ stretched along the diagonal,
and $\otimes$ means the rank-$1$ matrix format.

For discretization of the CME equation in time, we use the Crank-Nicolson scheme
with the step size $\tau$ and denote $t_k=\tau k $, $k=0,1,...,N_t$, $y_k=P_k \approx P(t_k)$,
and $f_k=f(t_k)$ in the case of nonzero right-hand side. We simplify the notations by setting
$N_i=N_x$ for $i=1,...,d$. Given $A, f_k, y_0$ in the TT/QTT format,
the system of discrete equations takes a form
\begin{equation*}\label{par:crank}
( I + \frac{\tau}{2} A ) y_{k+1} = ( I - \frac{\tau}{2} A) y_k + \frac{\tau}{2} (f_k+f_{k+1})=: F_{k+1},
\quad k=0,1,...,N_t.
 \end{equation*}

Two strategies suited for tensor methods can be applied:

{ (A)} Time stepping by DMRG-TT iteration for 
$$( I + \frac{\tau}{2} A )y_{k+1} =F_{k+1}.$$ 

{ (B)} Global {$O(\log N_x \log N_t)$} block solver in QTT format: 
$$
y_{k+1}-y_k + \frac{\tau}{2}A y_{k+1} + \frac{\tau}{2}A y_k = \frac{\tau}{2} (f_k+f_{k+1}).
$$

We follow the second approach that means
solving the huge global $N_x^d\times N_t$ system in QTT-Tucker format,
\begin{equation*}
\begin{bmatrix}
I + \frac{\tau}{2}A \\
-I + \frac{\tau}{2}A & I + \frac{\tau}{2}A \\
& \ddots & \ddots \\
& & -I + \frac{\tau}{2}A & I + \frac{\tau}{2}A
\end{bmatrix}
\begin{bmatrix}
y_1 \\ y_2 \\ \vdots \\ y_{N_t}
\end{bmatrix}
=
\begin{bmatrix}
\left(I - \frac{\tau}{2} A \right) y^0 \\ 0 \\ \vdots \\ 0
\end{bmatrix}
+ \frac{\tau}{2}
\begin{bmatrix}
f_0+f_1 \\ f_1+f_2 \\ \vdots \\ f_{N_t-1}+f_{N_t}
\end{bmatrix}.
\label{par:global_system}
\end{equation*}
The solution process can be based on either preconditioned iteration applied to symmetrized system
of equations or ALS-type iteration applied to the initial non-symmetric system.
Our numerical results are based on the second approach in the form of the so-called AMEn iteration
\cite{DoSav_AMEn:13}.

\vspace{-0.3cm}
\begin{center}
\begin{figure}[h!]
\begin{minipage}[t]{0.9\linewidth}
\centering
\caption{Cascade signaling network}
\label{fig:cascade}
\begin{tikzpicture}
\node [circle, draw] (s1) at (0,0) {$S_1$};
\node [circle, draw] (s2) at ($(s1)+(1.5,0)$) {$S_2$};
\node  (ccc) at ($(s2)+(1.2,0)$) {$\cdots$};
\node [circle, draw] (sd) at ($(ccc)+(1.2,0)$) {$S_d$};

\draw[sloped,semithick,->] (s1) to [out=60,in=120] (s2);
\draw[sloped,semithick,->] (s2) to [out=60,in=120] (ccc);
\draw[sloped,semithick,->] (ccc) to [out=60,in=120] (sd);
\draw[sloped,semithick,->] (s1) to [out=-135,in=180] ($(s1)+(0,-0.8)$) to [out=0,in=-45] (s1);
\draw[sloped,semithick,->] (s2) to [out=-135,in=180] ($(s2)+(0,-0.8)$) to [out=0,in=-45] (s2);
\draw[sloped,semithick,->] (sd) to [out=-135,in=180] ($(sd)+(0,-0.8)$) to [out=0,in=-45] (sd);
\end{tikzpicture}
\end{minipage}
%
\end{figure}
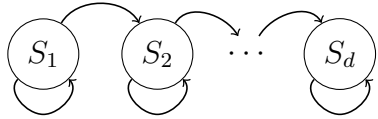
\end{center}
\vspace{-0.3cm}

We present numerical tests for high-dimensional cascade problem from 
\cite{Ammar-cme-2011,hegland-cme-2007}, 
which occurs when adjacent genes produce proteins which influence on the expression of a 
succeeding gene, see Fig. \ref{fig:cascade}. Tensor properties of this model, 
including tensor ranks of CME matrices,  
have been analyzed theoretically and numerically in \cite{DoKh_CME_NLAA:13}. 
\begin{itemize}
 \item $d=20$, $M=40$;
 \item for $m=1$: $w^{m}(\mathbf{x}) = 0.7$, $\mathbf{z}^{m} = -\delta_m$: generation 
 of the first protein;
 \item for $m=2,...,20$: $w^{m}(\mathbf{x}) = \dfrac{x_{m-1}}{5+x_{m-1}}$, 
 $\mathbf{z}^{m} = -\delta_m$: succeeding creation reactions;
 \item for $m=21,...,40$: $w^{m}(\mathbf{x}) = 0.07 \cdot x_{m-20}$, 
 $\mathbf{z}^{m} = \delta_{m-20}$: destruction reactions.
\end{itemize}
Here $\delta_m$ is the $m$-th identity vector,
$N_i=63$, hence the full grid problem size is about $64^{20}$.
The linear QTT format was used for state and time. The dynamical problem was solved 
until $T=400$ via the restarted global state-time solver after each $T_0=15$ time sub-intervals.
This means that the global state-time solver applies to  $(x,t)$ discretization 
successively on each coarse time interval of size $T_0=15$, so that number of intervals
is $T/T_0$. 
The initial state was chosen by $P(0)=\delta_1 \otimes \cdots \otimes \delta_1$, i.e. 
all copy numbers are zeros. The solution threshold in the discrete $L_2$norm 
for the QTT-rank truncation was chosen as $\varepsilon= 10^{-5}$.

Fig. \ref{fig:20d:meanconc} illustrates the $\log N_t$ scaling of the solution time 
for the varying number of time steps, $N_t=2^8,2^9,...,2^{14}$.
Fig. \ref{fig:20d:|Au|} shows the convergence of the transient solution to stationary one.
This confirms that the chosen time interval $T$ is large enough (long-time dynamics)
to catch the steady state with the required precision.

\begin{center}
\begin{figure}[h!]
\begin{minipage}[t]{0.4\linewidth}
\centering
\includegraphics[width=\linewidth]{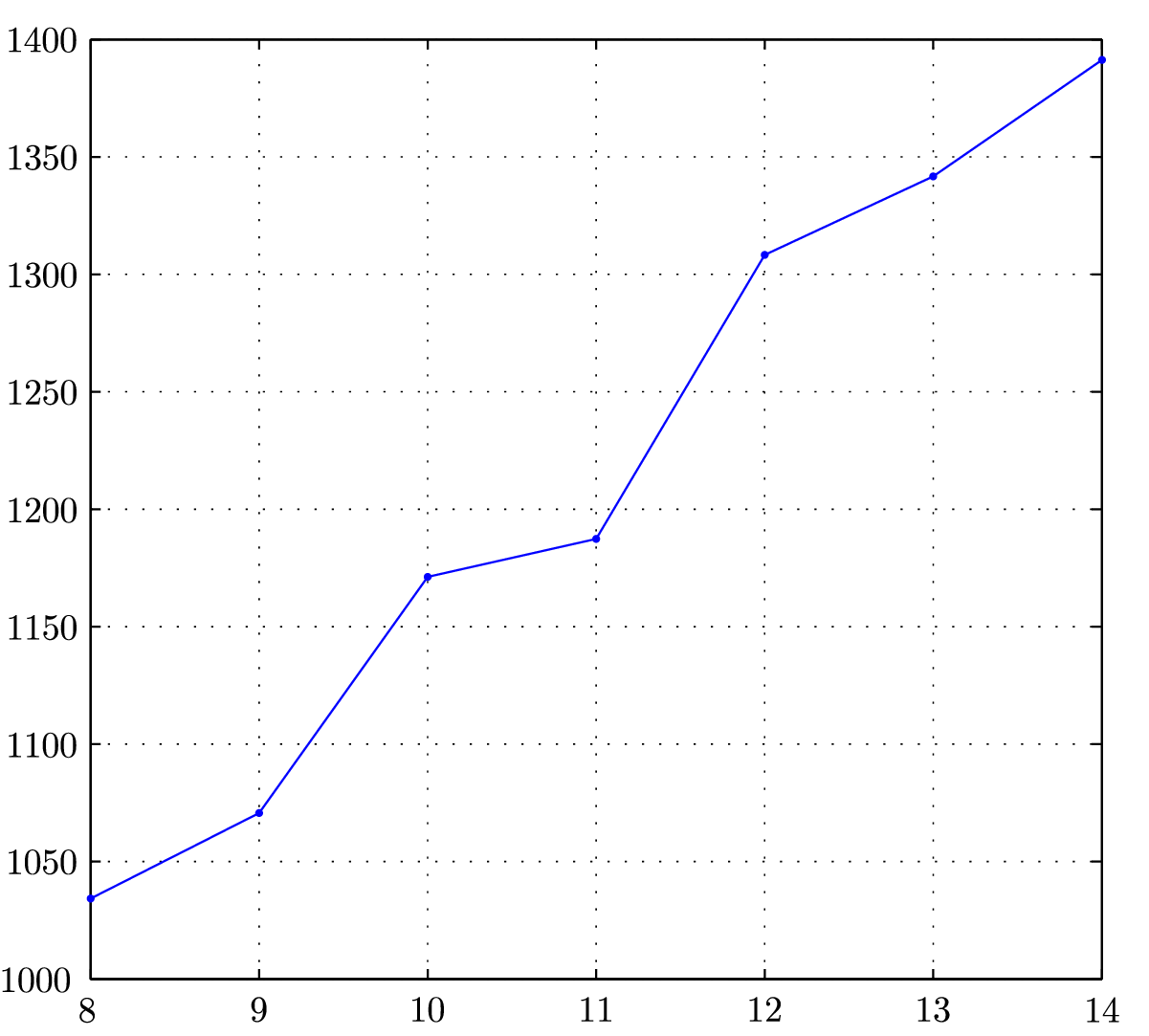}
\caption{\tiny CPU time (sec.) vs. $\log N_t$.}
\label{fig:20d:meanconc}
\end{minipage}
\quad \quad
\begin{minipage}[t]{0.4\linewidth}
\centering
\includegraphics[width=\linewidth]{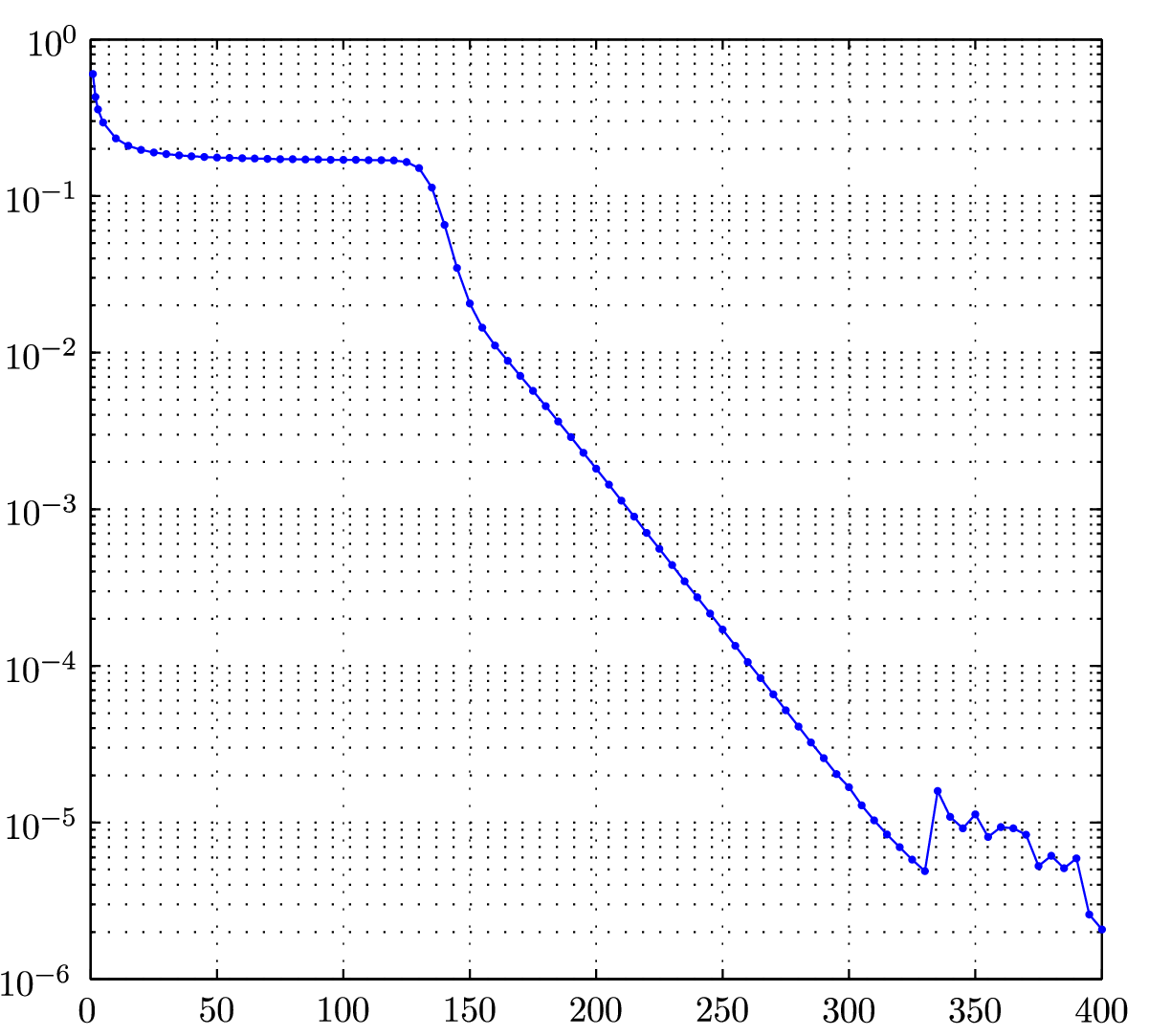}
\caption{\tiny Closeness to the kernel $\dfrac{||AP||}{||P||}(t)$}
\label{fig:20d:|Au|}
\end{minipage}
\end{figure}
\end{center}

We conclude that presented results demonstrate the high performance of QTT- and QTT-Tucker
formats in tensor computations for multi-dimensional CMEs.
For many interesting cases the theoretical rank bounds for the CME matrices have been proven
\cite{DoKh_CME_NLAA:13}.
Numerical tests confirm the logarithmic complexity scaling of the global space-time
tensor schemes in both time and space discretization parameters,
indicating the potential advantages of this approach for high-dimensional dynamical simulations.

\section{Conclusions}\label{sec:Conclusion}

The main ingredients and favorable features of the modern {\it tensor numerical methods} in applications to 
multidimensional PDEs have been discussed.
We focused on the recent finding of  quantized tensor approximation
that allows to represent discrete multivariate functions and operators on 
large $d$-dimensional grids of size $N^d$, with log-volume complexity, $O(d \log N)$.
We show how this approach can be applied to the 
low-parametric representation of functions and operators 
on the examples of nonlinear Hartree-Fock and real-time dynamical master equations.

Other examples on successful applications of tensor numerical methods
to parametric/stochastic PDEs, time-dependent parabolic equations, 
multidimensional eigenvalue problems, the fast lattice summation schemes, superfast data 
transforms and to integration of singular and highly oscillating functions can be found in 
references provided.


\vspace{0.5cm}
{\bf Acknowledgements.} The author gratefully acknowledge Dr. V. Khoromskaia 
(MPI MiS Leipzig) for useful comments on Section \ref{sec:TS-formats} 
and \S\ref{ssec:HaFockeqn} which have led to the substantial improvement 
of the revised manuscript.

%
% ...
%


\end{document}